\documentclass[reqno,12pt]{amsart} 
\usepackage{amsmath,amssymb,amsthm,enumerate,mathrsfs,colonequals}
\usepackage{fullpage,url,tikz,xspace,setspace}
\usepackage{microtype}
\usepackage{calc}

\usepackage[
unicode=true
]{hyperref}
\usepackage[alphabetic,lite,nobysame]{amsrefs} 
\usepackage{amsfonts,varioref,mathtools}
\usepackage{latexsym}
\usepackage{amscd,graphicx,float,xcolor, enumerate}
\usepackage[all]{xy}
\usepackage{verbatim, tikz, rotate}
\usepackage{etoolbox,amsthm}
\usepackage{svg}
\DefineSimpleKey{bib}{myurl}
\newcommand\myurl[1]{\url{#1}}
\BibSpec{webpage}{%
  +{}{\PrintAuthors} {author}
  +{,}{ \textit} {title}
  +{}{ \parenthesize} {date}
  +{,}{ \myurl} {myurl}
  +{,}{ } {note}
  +{.}{ } {transition}
}

\title{Quotient graphs and amalgam
  presentations for unitary groups over
  cyclotomic rings}
\author{Colin Ingalls, Bruce W. Jordan, Allan Keeton, \\ Adam Logan,  and Yevgeny Zaytman}

\address{School of Mathematics and Statistics, Carleton University, Ottawa, ON K1S 5B6, Canada}
\email{cingalls@math.carleton.ca}

\address{Department of Mathematics, Box B-630, Baruch College,
The City University of New York, One Bernard Baruch Way, New York
NY 10010}
\email{bruce.jordan@baruch.cuny.edu}

\address{Center for Communications Research, 805 Bunn Drive,
Princeton, NJ 08540}
\email{agk@idaccr.org}

\address{The Tutte Institute for Mathematics and Computation,
  P.O. Box 9703, Terminal, Ottawa, ON K1G 3Z4, Canada}
\address{School of Mathematics and Statistics, 4302 Herzberg
  Laboratories, 1125 Colonel By Drive, Ottawa, ON K1S 5B6, Canada}
\email{adam.m.logan@gmail.com}

\address{Center for Communications Research, 805 Bunn Drive,
Princeton, NJ 08540}
\email{ykzaytm@idaccr.org}


\DeclareMathAlphabet{\mathcalligra}{T1}{calligra}{m}{n}
\newcommand{\allan}[1]{{\color{blue}\sf [Allan: #1]}}
\newcommand{\bruce}[1]{{\color{purple}\sf [Bruce: #1]}}

\newcommand{\adam}[1]{{\color{green}\sf [Adam: #1]}}

\newcommand{\gr}{\mathit{gr}}
\newcommand{\Gr}{\mathit{Gr}}
\newcommand{\grn}{\gr\!_n}
\newcommand{\Grn}{\Gr\!_n}
\newcommand{\oGr}{\,\protect\overline{\! \Gr }}
\newcommand{\ogr}{\,\protect\overline{\! \gr }}
\newcommand{\oGrn}{\oGr\!_n}
\newcommand{\ogrn}{\ogr\!_n}

\newcommand{\dbs}{\hspace{0.013in}\backslash\hspace{-0.03in}
  \backslash\hspace{-0.02in}}
\newcommand{\Z}{{\mathbf Z}}

\newcommand{\Grz}{\gr_{0}}

\newcommand{\Grp}{\gr_{+}}
\newcommand{\Gro}{\gr_{1}}

\newcommand{\A}{{\mathbb A}}
\newcommand{\BC}{{\mathbf C}}

\newcommand{\FAAv}{{F^{\times,v}_{\A}}}
\newcommand{\FAAa}{{F^{\times}_{\A}}}
\newcommand{\BR}{{\mathbf R}}
\newcommand{\F}{{\mathbf F}}

\renewcommand{\H}{{\mathbf H}}
\newcommand{\M}{{\mathcal M}}

\newcommand{\sM}{{\mathscr M}}
\newcommand{\Sc}{{\mathcal S}}

\newcommand{\Gg}{{\mathcal G}}
\newcommand{\PGg}{{\mathrm P}{\mathcal G}}
\newcommand{\SGg}{{\mathrm S}{\mathcal G}}
\newcommand{\PSGg}{{\mathrm P}{\SGg}}

\newcommand{\Pic}{{\rm Pic}\, }
\newcommand{\Q}{{\mathbf Q}}
\newcommand{\OO}{{\mathcal O}}

\newcommand{\OAAv}{{\OO^{\times, v}_{\A}}}

\newcommand{\T}{{\mathscr T}}
\newcommand{\p}{{\mathfrak p}}
\newcommand{\fa}{{\mathfrak a}}

\newcommand{\fp}{{\mathfrak p}}
\newcommand{\fq}{{\mathfrak q}}
\newcommand{\fP}{{\mathfrak P}}

\newcommand{\be}{{\mathbf e}}

\newcommand{\bv}{{\mathbf v}}
\newcommand{\bw}{{\mathbf w}}

\newcommand{\cyc}[1]{{K_{#1}}}
\newcommand{\realcyc}[1]{{K_{#1}^+}}
\newcommand{\cycthree}[1]{{L_{#1}}}

\newcommand{\uOO}{\,\protect\underline{\! \OO \!}\,}
\newcommand{\uOOn}{\uOO_n}

\newcommand{\uR}{\underline{\! R \!}\,}
\newcommand{\uRn}{\uR_n}
\newcommand{\ovn}{\overline{\varphi}_n}

\DeclareMathOperator{\VM}{VM}
\DeclareMathOperator{\EM}{EM}

\DeclareMathOperator{\UT}{U_{2}}
\DeclareMathOperator{\UTz}{U_{2}^{\zeta}}
\DeclareMathOperator{\PUT}{PU_{2}}
\DeclareMathOperator{\PUTz}{PU_{2}^{\zeta}}
\newcommand{\PGL}{{\rm PGL}}
\DeclareMathOperator{\PGLT}{PGL_{2}}

\DeclareMathOperator{\PSUT}{PSU_{2}}
\DeclareMathOperator{\SUT}{SU_{2}}

\DeclareMathOperator{\m}{m}
\DeclareMathOperator{\SLT}{SL_{2}}

\newcommand{\pmat}[1]{\begin{pmatrix}#1\end{pmatrix}}
\newcommand{\isisom}{\cong}

\newcommand{\hrel}{h_{\rm rel}}
\newcommand{\hnar}{\tilde{h}}

\newcommand{\sinf}{\ll}
\newcommand{\CC}{\mathcal{C}}

\newcommand{\mO}{{\mathcal O}}

\newcommand{\f}{{\mathfrak f}}

\DeclareMathOperator\Verts{Ver}

\DeclareMathOperator\SL{SL}

\DeclareMathOperator\PU{PU}

\DeclareMathOperator\HH{H}
\DeclareMathOperator\rank{rank}

\DeclareMathOperator\Mattt{\Mat_{2\times 2}}

\newtheorem{theorem}{Theorem}

\newtheorem{prop}[theorem]{Proposition}

\theoremstyle{definition}
\newtheorem{definition}{Definition}
\newtheorem{definition1}[theorem]{Definition}
\theoremstyle{remark}
\newtheorem{remark1}[theorem]{Remark}
\theoremstyle{hypothesis}
\newtheorem{hypothesis}[theorem]{Hypothesis}
\theoremstyle{theorem}
\newtheorem{notation}[theorem]{Notation}
\newtheorem{key}[theorem]{Example Key}
\theoremstyle{key}

\DeclareMathOperator\Gal{Gal}

\DeclareMathOperator\Pro{P}

\DeclareMathOperator\Nm{N}

\DeclareMathOperator\Mat{Mat}
\DeclareMathOperator\Disc{Disc}
\DeclareMathOperator\Norm{Norm}

\DeclareMathOperator\Ver{Ver}
\DeclareMathOperator\Ed{Ed}
\DeclareMathOperator\dist{dist}

\DeclareMathOperator\Cl{Cl}

\DeclareMathOperator\Cln{\widetilde{Cl}}
\DeclareMathOperator\val{val}

\DeclareMathOperator\sig{sig}
\DeclareMathOperator\Coker{Coker}

\DeclareMathOperator\Sel{Sel}
\DeclareMathOperator\Prin{Prin}
\DeclareMathOperator\Prinp{\Prin_{+}}

\DeclareMathOperator\Id{Id}

\makeatletter
\patchcmd{\@part}{\null\vfil}{}{}{}
\patchcmd{\@part}{\par}{.\,\,}{}{}
\makeatother

\newcommand{\Gn}{\mathcal{G}_n}
\DeclareMathOperator\genus{\sf g}

\newtheorem{example1}[theorem]{Example}
\theoremstyle{sa}

\theoremstyle{mtheorem}
\theoremstyle{remark}

\DeclareMathOperator\Pp{P}
\newcommand{\PGn}{\Pp\!\Gn}
\newcommand{\PG}{\Pp\!\Gg}

\let\widebar=\overline

\definecolor{fuchsia}{RGB}{255,0,255}

\begin{document}

\maketitle

\begin{abstract}
  Suppose $4|n$, $n\geq 8$, 
  $F=F_n=\Q(\zeta_n+\overline{\zeta}_n)$, and there is one prime
  $\fp=\fp_n$ above $2$ in $F_n$.
  We study amalgam presentations
  for $\PUT(\Z[\zeta_n, 1/2])$ and $\PSUT(\Z[\zeta_n, 1/2])$
  with the Clifford-cyclotomic group in quantum
  computing as a subgroup. These amalgams arise from
  an action of these groups on the Bruhat-Tits tree
  $\Delta =\Delta_{\fp}$ for $\SLT(F_\fp)$
  constructed via the Hamilton quaternions.
  We explicitly compute the finite quotient graphs and the
  resulting amalgams for $8\leq n\leq 48$, $n\neq 44$, as well
  as for $\PUT(\Z[\zeta_{60}, 1/2])$.  
\end{abstract}

\clearpage

\tableofcontents
\clearpage

\section{Introduction}

This is the third in a series of papers devoted to
the structure of unitary groups over cyclotomic rings.
The first of these papers \cite{IJKLZ} concerned the
Euler-Poincar\'{e} characteristic of these groups.
This invariant was sufficient, following Serre, to prove a
conjecture of Sarnak \cite[p.~$15^{\rm IV}$]{S} on
when these groups are generated by the Hadamard gate and the T-gate---two 
specific elements of finite
order \cite[Theorem 1.2]{IJKLZ}.
The second paper \cite{IJKLZ2} analyzed the corank of these
groups, a more difficult invariant than the Euler-Poincar\'{e}
characteristic, but only in the families
$n=2^s$ and $n=3\cdot 2^s$ where simplifications occur.
In this paper we consider the case of general $n$, subject
to the standing assumption that $n=2^sd$, $d$ odd,
$s\geq 2$, $n\geq 8$, and Hypothesis \ref{assume} below:
$\langle 2, -1\rangle =(\Z/d\Z)^\times$.
Here we  continue the method of \cite{IJKLZ2},
analyzing an action of
these groups on Bruhat-Tits trees $\Delta$ together with the
resulting
finite
quotient graphs, with the emphasis on computing examples.

Set $\zeta_n=e^{2\pi i/n}$.
The cyclotomic field $K_n:=\Q(\zeta_n)$
has integers $\OO_n:=\Z[\zeta_n]$ and totally real subfield
$F_n:=K_n^+=\Q(\zeta_n+\overline{\zeta}_n)$ with integers
$\uOOn:=\Z[\zeta_n]^+=\Z[\zeta_n+\overline{\zeta}_n]$. We 
set $R_n:=\OO_n[1/2]$ and $\,\uRn:=R_n^+=\uOO_n[1/2]$.
 By our assumption on $n$,
  the cyclic group of roots of unity in $K_n$ is generated by
  $\zeta_n$ and contains $i$.  Also $F_n\neq \Q$ and the $\uOO_n$-ideal
  $(2)$ is the square of an ideal of $\uOO_n$, which we will denote by
  $\fq=\fq_n$.
  Let $\H$ be the Hamilton quaternions over $\Q$ (the rational quaternion
  algebra ramified precisely at $2$ and $\infty$), and put
  $\H_n=\H\otimes_{\Q} F_n$.
We fix a $\Q$-basis $1$, $i$, $j$, $k$ of
  $\H$ satisfying $i^2=j^2=k^2=-1$, $ij=-ji$, $ik=-ki$, $jk=-kj$.
 The {\em standard} maximal 
 $\uRn$-order of $\H_n$ is
 \[
 \widetilde{\M}_n:=\uRn\langle 1, i, j, (1+i+j+k)/2\rangle.
 \]

Define the Hadamard matrix $H$ and the matrix $T_n$ by 
\begin{equation}
  \label{guestroom}
H:= \frac{1}{2}\left[\begin{array}{rr}
    1+i  & 1+i\\
    1+i & -1-i\end{array}\right]
    \quad\mbox{and}\quad
T_n:=\left[\begin{array}{rr}
    1 & 0\\
    0 & \zeta_n\end{array}\right];
  \end{equation}
we have $H,T_n\in \UT(R_n)$.
The 
{\em Clifford-cyclotomic group} \cite[Section 2.2]{FGKM} (resp.,
{\em special} Clifford-cyclotomic group) is
\begin{equation}
\label{tooth}
\Gg_n=\langle H, T_n\rangle\qquad\mbox{(resp., $\SGg_n=\Gg_n\cap \SUT(R_n)$)}.
\end{equation}
Put
\begin{equation}
\label{groom}
\UTz(R_n)=\{\gamma \in\UT(R_n)\mid \det\gamma\in\langle\zeta_n\rangle\};
\end{equation}
we then have $\Gg_n\subseteq \UTz(R_n)\subseteq \UT(R_n)$.
In general, $\UTz(R_{n})\subsetneq \UT(R_n)$.

Various subgroups and quotient groups of
$\UT(R_n)$ and $\SUT(R_n)$ occur throughout this
paper.
It is convenient to use the following notation:
\begin{notation}
\label{grits}
\end{notation}
\vspace*{.1in}
{\rm
\begin{center}
    \begin{tabular}{c|l}
      $H\leq G$& $H$ is a subgroup of $G$\\ \hline
      $H\unlhd G$ & $H$ is a {\em normal} subgroup of $G$\\ \hline
      $H\ll G$ & $H\leq G$ and $[G:H]=\infty$\\ \hline
      $H\!\!<\!\!\!\!\! \lhd\, G$ & $H\unlhd G$ and $[G:H]=\infty$  \\ \hline
      $H\lesssim G$ & $H\leq G$ and $[G:H]<\infty$\\ \hline
     $H\,\raisebox{-.75ex}{$\stackrel{\mbox{$\lhd$}}{\sim}$}\,G$ & $H\unlhd G$ and $[G:H]<\infty$ 
 \end{tabular}
\end{center}
}
\vspace*{.4in}
\noindent For $H\leq \UT(R_n)$  denote by
$\Pro\!H$ the image of $H$ in $\PUT(R_n)$.
For $H\leq \H_n^\times$, put $H_1=\{h\in H\mid \Nm_{\H_n/F_n}(h)=1\}$; we
have $H_1\unlhd H$.
For a group $G$, denote by $G_f\unlhd G$ the (normal) subgroup 
generated by the elements of $G$ of finite order.
We have the subgroup structure
\begin{align}
\Gg_n\leq \UT(R_n)_f\unlhd & \UTz(R_n)\unlhd \UT(R_n), \quad
\PUTz(R_n)\,\raisebox{-.75ex}{$\stackrel{\mbox{$\lhd$}}{\sim}$}\,\PUT(R_n), 
\quad\mbox{and}\\
\nonumber
 & \SGg_n\leq\SUT(R_n)_{f}\unlhd\SUT(R_n).
  \end{align}
If \mbox{$\UTz(R_n)\neq\UT(R_n)$},
then $\UTz(R_n)\!\!<\!\!\!\!\! \lhd\,\UT(R_n)$.
The structure of $\PGn$ is known from \cite[Theorem 1]{RS}; see
\cite[Theorem 4.1]{IJKLZ}.
\begin{theorem}[Radin and Sadun]
\label{bird}
Let $S_4$ be the symmetric group on $4$ letters and $D_m$ be the
dihedral group of order $2m$.  Then
$\PGn\simeq S_4\ast_{D_4}D_n$.
\end{theorem}

For certain $n$ there is a natural action of $\UT(R_n)$ and
$\SUT(R_n)$ on a Bruhat-Tits tree $\Delta$ with finite stabilizers
and finite quotient graph. The condition on $n$ for these
finite quotient graphs to exist is:
\begin{hypothesis}
\label{assume}
$\langle 2, -1\rangle =(\Z/d\Z)^\times$ .
\end{hypothesis}
\noindent Hypothesis \ref{assume} implies the following:
\begin{enumerate}[\upshape (a)]
\item
\label{assume1}
There is one prime $\fp=\fp_n$ of $F=F_n$ above $2$ and
$\H=\H_n$ is unramified at $\fp$.
\item
\label{assume2}
There are explicit embeddings
\[
\varphi_n:\PSUT(R_n)\stackrel{\simeq}{\twoheadrightarrow} \Gamma_n\subseteq
\Pp\!\H_{n,1}^\times\quad\text{and}\quad
\overline{\varphi}_n:\PUT(R_n)\stackrel{\simeq}{\twoheadrightarrow}
\overline{\Gamma}_n\subseteq
\Pp\!\H_n^\times
\]
with $\overline{\varphi}_n|_{\PSUT(R_n)}=\varphi_n$
and $\Gamma_n=\Pp\!\widetilde{\M}_{n,1}^\times=\widetilde{\M}_{n,1}/\langle
\pm 1\rangle$, see Section \ref{fiend}.
\item
\label{assume3}
Let $\Delta=\Delta_\fp$ be
the Bruhat-Tits tree for $\SLT(F_{\fp})$.  Then $\PGLT(F_{\fp})$
acts on $\Delta$.  The identifications $\varphi_n$ and $\overline{\varphi}_n$
above give an action of $\PSUT(R_n)$ and $\PUT(R_n)$ on $\Delta$.
There are finite quotient graphs $\grn=\Gamma_n\backslash\Delta$ and
$\ogrn=\overline{\Gamma}_n\backslash\Delta$.
Moreover the stabilizers $\Gamma_{n,\bv}$ and $\overline{\Gamma}_{n,\bv}$
of a vertex $\bv\in\Ver(\Delta)$ in $\Gamma_{n}$ and $\overline{\Gamma}_n$,
respectively, are finite.  Likewise the stabilizers $\Gamma_{n,\be}$
and $\overline{\Gamma}_{n,\be}$ are finite for an edge $\be\in\Ed(\Delta)$.
More generally there are quotient graphs-of-groups
$\Grn=(\Gamma_n,\grn)$ and quotient h-graphs-of-groups 
$\oGrn=(\overline{\Gamma}_n,\ogrn)$.
Knowing $\Grn$ and $\oGrn$ 
gives amalgam presentations of $\PSUT(R_n)\cong\pi_1(\Grn)$ and
$\PUT(R_n)\cong\pi_1(\oGrn)$ as
in~\cite{S2} and Section \ref{groupgraph} of this paper.
\end{enumerate}
\noindent If Hypothesis \ref{assume} is {\em not} satisfied, then
instead of quotient graphs one gets quotient regular cubical
complexes of dimension $d\geq 2$ as in \cite{JL}.
We do not treat these higher-dimensional quotients here.
The first $n$ for which Hypothesis \ref{assume} fails is
$n=68$.

The initial part of this paper, Sections \ref{groupgraph}--\ref{fiend},
establishes the theoretical
foundations for computing examples.  Much of this material
extends the results in \cite{IJKLZ2} for the specific
families $n=2^s$ and $n=3\cdot 2^s$ to general $n$.
The highlights of the paper are in the second part, Sections
\ref{part: examples}--\ref{summary}, where 
we  compute $\Grn$ and $\oGrn$ in MAGMA \cite{BCP}  with
corresponding amalgam presentations
for $\SUT(R_n)$ and $\PUT(R_n)$ for $8\leq n\leq 48$, $4|n$, $n\neq 44$.
We give the quotient h-graph of groups $\oGr_{60}$
and the corresponding amalgam
presentation for $\PUT(R_{60})$.

A surprising feature of the examples is that we are able to identify
$\PGn$ as the fundamental group of a sub h-graph-of-groups of
$\Grn$.  Subgroups of amalgamated products are not in general
sub amalgamated-products.  But we get an amalgamated product
presentation of $\pi_1(\Grn)\cong\PUT(R_n)$ with $\PGn\cong S_4\ast_{D_4}D_n$
as a sub amalgamated-product.

Here is part of the $n=28$ example (see Section \ref{rounded} for the definitions
of the groups):
\begin{example1}
\label{apple}
\textup{ (Section \ref{event}) }
{\rm 
Let $A_m$ (resp., $S_m$)
denote the alternating group (resp., symmetric group) on $m$ letters,
$C_m$ the cyclic group of order $m$, $D_m$ the dihedral group of order
$2m$, and $Q_{2m}$ the quaternion group of order $2m$.
Denote the binary tetrahedral and octahedral groups by $E_{24}$ and
$E_{48}$, respectively.
\begin{enumerate}[\upshape (a)]
\item
$\PUT(R_{28})\cong D_{28}\ast_{C_{28}}D_{28}\ast_{D_4}S_{4}\ast C_2^{\ast 2}
=D_{28}\ast_{C_{28}}\PG_{28}\ast C_2^{\ast 2}$.
\item
$\PGg_{28}\ll [\PUT(R_{28})]_{f}=\PUT(R_{28})$.
\item
$\PUTz(R_{28})\cong S_{4}\ast_{D_4}D_{28}\ast_{C_{28}}D_{28}\ast_{D_4}S_4\ast
\Z^{\ast 2}$ and
\begin{equation*}
\label{wine1}
\Gg_{28}\ll U_{2}(R_{28})_f\!\!<\!\!\!\!\! \lhd\,   \UTz(R_{28})
\!\!<\!\!\!\!\! \lhd\, \UT(R_{28}).
\end{equation*}
\item
$\SUT(R_{28})\cong E_{48}\ast_{Q_8}Q_{56}\ast_{C_{28}}Q_{56}\ast_{Q_8}E_{48}
\ast \Z^{\ast 4}$ and
\begin{equation*}
\label{wine2}
\SGg_{28}\ll \SUT(R_{28})_f
\!\!<\!\!\!\!\! \lhd\, \SUT(R_{28}).
\end{equation*}
\end{enumerate}
Theorem 1.2 of \cite{IJKLZ} already showed that
$\Gg_{28}\ll \UT(R_{28})$ and $\SGg_{28}\ll\SUT(R_{28})$.
However, the explicit presentations  and the further
subgroup results above are new.
}
\end{example1}

\section{H-graphs and h-graphs of groups}
\label{groupgraph}

The standard reference for 
graphs constructed as  quotients of trees
by group actions is Serre's book \cite{S2}.  The generalization to
h-graphs by Kurihara \cite{K}
is treated in \cite[Section 1]{IJKLZ2}, which
we use freely along with \cite{S2}.
Following \cite{S2}, a {\em graph}
has oriented edges $\be$ along with their opposites $\bar\be$, which
are distinct.
In an {\em h-graph}, the definition is
relaxed to allow 
{\it half-edges}, edges $\be$
with $\be=\bar\be$, as in \cite{K}.
Edges $\be$ with
$\overline{\be} \ne \be$ are {\it regular edges}. Write $\Ed_r(\gr)$ and
$\Ed_h(\gr)$ for the collection of regular and half-edges of $\gr$
respectively and $\Ed(\gr) := \Ed_r(\gr) \amalg \Ed_h(\gr)$ for the
set of all edges.  Half-edges $\be$ originate and terminate
at the same vertex $o(\be)=t(\be)$. Every graph is also an 
h-graph.

Suppose $\gr$ is a finite connected h-graph with vertices
$\Ver(\gr )$ and 
$v=v(\gr)
=\#\Ver(\gr)$.
 Set $e_r(\gr)  = \frac{1}{2}\#\Ed_r(\gr)$, 
 $e_h =
 \frac{1}{2}\# \Ed_h(\gr)$, and
$e=e(\gr):= e_r(\gr) + e_h(\gr)$. The fundamental
group
$\pi_1(\gr)$ has abelianization isomorphic
to $H_1(\gr, \Z)$.
The genus 
$\genus(\gr)$ of $\gr$ is the first Betti number $\rank\HH_1(\gr, \Z)$.
By Euler's formula
$\genus(\gr)=1+e_r-v$.

\begin{definition1}
A {\it graph of groups} \cite[Section 5]{S2} is a pair $\Gr = (\Gamma,
\gr)$ with $\gr$ a graph and $\Gamma$ an assignment $\bv \mapsto \Gamma_\bv$,
  $\be\mapsto \Gamma_\be$ of a
group to each
$\bv\in\Ver(\gr)$, $\be\in\Ed(\gr)$
with $\Gamma_{\bar{\be}} = \Gamma_\be$ together with an injection $\Gamma_\be \hookrightarrow \Gamma_{t(\be)}$ (denoted
$g \mapsto g^\be$). For an edge $\be\in\Ed(\gr)$
we have injections $\Gamma_\be \hookrightarrow
\Gamma_{t(\be)}$ and $\Gamma_\be = \Gamma_{\bar{\be}} \hookrightarrow \Gamma_{t(\bar{\be})} = \Gamma_{o(\be)}$
into the vertex groups of the origin and target vertices. The
first sends $g \in \Gamma_\be$ to $g^\be \in \Gamma_{t(\be)}$ and the latter to
$g^{\bar{\be}} \in \Gamma_{o(\be)}$. 
\end{definition1}

A graph of groups $\Gr = (\Gamma,\gr)$ has a fundamental group
$\pi_1(\Gamma, \gr) = \pi_1(\Gr)$ with a surjection 
\begin{equation}
\label{pi_proj}
\pi_1(\Gr) \twoheadrightarrow \pi_1(\gr)
\end{equation}
whose kernel is the normal closure of $\langle \Gamma_\bv : \bv \in \Ver(\gr)\rangle$.

If $\gr$ is connected and acyclic then $\pi_1(\Gr) = \pi_1(\Gamma,
\gr)$ is the amalgamation of the vertex groups over the edge groups.
If $\gr$ is connected but not acyclic, 
choose a spanning tree $T$ by deleting a collection of
edges $E \subset \Ed(\gr)$. Let $(\Gamma, T)$ be the associated subgraph
of groups. The fundamental group of $\Gr$ {\it based at $T$},
$\pi_1(\Gr; T)$, is defined to be the group generated by
$\pi_1(\Gamma, T)$ together with generators
$\{x_{\be}: \be \in E\}$ subject only to the relations that $g^{\be}
x_{\be} = x_\be g^{\bar{\be}}$ for $g \in \Gamma_\be$.

We denote  the free product of the groups $A$ and $B$ by $A\ast B$
with $A^{\ast 1}:=A$ and $A^{\ast n+1}:=A\ast A^{\ast n}$ for $n\geq 1$.
So $\Z^{*n}$ is the free group on $n$ generators.
Since $\# E =
\genus (\gr)$ we  have that 
\begin{equation}
\pi_1(\Gr, T) \simeq \pi_1(\Gamma, T)*\Z^{*{\genus(\gr)}}/\mathcal{R}\, .
\end{equation}
with $\mathcal{R}$ the relations on the $x_{\be}$ described above. It is a
theorem~\cite[Proposition I.20]{S2} that the isomorphism class of 
$\pi_1(\Gr; T)$ does not depend on the choice of spanning tree $T$; we 
therefore denote it by $\pi_1(\Gr)$.
There is a construction for $\pi_1(\Gr)$ that
does not require fixing a spanning tree, but for our purposes (explicit
representations of $\pi_1(\Gr)$ as amalgamated products) choosing a
spanning tree is more convenient. Adjoining the
generators $x_{\be}$ one at a time constructs $\pi_1(\Gr)$ as a
$\genus(\gr)$-fold iterated HNN extension of the amalgam $\pi_1(\Gr, T)$.

A group $\Gamma$ acting on a graph $\gr$ determines a graph of groups
$\Gr := (\Gamma, \gr)$ by assigning the stabilizer group $\Gamma_\be$
or $\Gamma_\bv$ in $\Gamma$ of an edge $\be$ or a vertex $\bv$. If
$\Gamma$ acts without inversions on $\gr$ there is an induced quotient
graph of groups $\Gamma\dbs\Gr$ with underlying graph $\Gamma
\setminus \gr$ defined as follows.  Let $\gr \xrightarrow{\pi} \Gamma
\backslash \gr$ be the quotient map.  For $\bv\in\Ver(\Gamma\backslash
\gr)$ choose $\tilde{\bv}\in\Ver(\gr)$ lying above $\bv$ and set
$\Gamma_{\bv}:= \Gamma_{\tilde{\bv}}$.  Similarly, for
$\be\in\Ed(\Gamma\backslash \gr)$ choose $\tilde{\be}\in\Ed(\gr)$ with
$\pi(\tilde{\be})=\be$ set $\Gamma_{\be}:= \Gamma_{\tilde{\be}}$.  Also choose
an element $g \in \Gamma$ with $t(g\cdot \tilde{\be}) =
\widetilde{t(\be)}$ and define
the injection $\Gamma_{\be}\rightarrow \Gamma_{t(\be)}$ as the
composition
$$ \Gamma_{\be} = \Gamma_{\tilde{\be}} \xrightarrow{g \cdot}
\Gamma_{g\cdot\tilde{\be}} \rightarrow \Gamma_{t(g\cdot \tilde{\be})}
= \Gamma_{\widetilde{t(\be)}} = \Gamma_{t(\be)}\, . $$ The choices of
$\tilde{\be}$, $\tilde{\bv}$, and $g$ are arbitrary, subject to the
above constraints, but once chosen are fixed.  Notice that the maps
$\Gamma_{\be}\rightarrow \Gamma_{t(\be)}$ are only well-defined up to
conjugation by elements of $\Gamma_{t(\be)}$.

\begin{remark1}\label{inj pi}
  {\rm Note that if $\Gr'\subset\Gr$ is a subgraph with all groups given by
  pullback, then there exists a natural injection
  $\pi_1(\Gr')\to\pi_1(\Gr)$.}
\end{remark1} 

The following is a key result in Bass-Serre theory:
\begin{theorem}[~\cite{S2}]\label{theorem: reconstruct}
Let $\Gamma$ be a group which acts without inversions on a tree
$\Delta$ and let $\Gr := \Gamma\dbs(\Gamma, \Delta)$ be the associated
quotient graph of groups.  Then $\Gamma \simeq \pi_1(\Gr)$.
\end{theorem}

If $\Gamma$ acts on a graph $\gr$ with inversions, let
$\gr_{\Gamma}$ be the graph obtained from $\gr$ by subdividing exactly those edges
that are inverted by $\Gamma$. By a {\it quotient h-graph of groups} for $\Gamma$
acting on $\gr$ we mean $\Gamma\dbs(\Gamma,
\gr_{\Gamma})$. When drawing h-graphs of groups we label each
vertex and edge with its stabilizer group. We also draw all pairs $\{\be, \bar{\be}\}$ as a
single undirected edge.
In order to make clear the h-graph structure coming from
$\Gamma\setminus \gr$ we elide the extra vertices coming from the barycentric
subdivision of inverted edges in an h-graph of groups. We only label
the stabilizer subgroup of
the corresponding half-edge if it differs from that of the elided
vertex. An example is: 

\begin{center}
  \begin{tikzpicture}
\draw[ultra thick,color = green] (0.4,0)--(2.6,0);
\draw[ultra thick,color = green] (3.4,0)--(4.5,0);
\draw[ultra thick,color = blue] (0, 0) circle [radius = 0.4];
\draw[ultra thick,color = blue] (3, 0) circle [radius = 0.4];
\node at (0,0){$\Gamma_{{\bf v}_0}$};
\node at (1.5,0.3){$\Gamma_{{\be_{0,1}}}$};
\node at (3,0){$\Gamma_{{\bf v}_1}$};
\node at (5,0){$\Gamma_{\bf v}$};
\end{tikzpicture}\,\,\, ,
\end{center} 
where $\bv$ is a vertex lying under a vertex associated to the
barycenter of an inverted edge in the tree.

Now suppose that $\Gamma$ acts with inversions on a tree $\Delta$.
By Theorem \ref{theorem:
  reconstruct}, we know that $\Gamma \simeq \pi_1(\Gamma\dbs(\Gamma,
\Delta_{\Gamma}))$.  If $\Gamma_0
\subset \Gamma$ acts on $\Delta$ without inversions, then we have a
cover $$\Gamma_0 \setminus \Delta \rightarrow \Gamma \setminus
\Delta $$
of an h-graph by a graph. We also have the cover of graphs
$$\Gamma_0 \setminus \Delta_{\Gamma} \rightarrow \Gamma \setminus
\Delta_{\Gamma}\, $$
with the induced group injection
$$\pi_1(\Gamma_0\dbs (\Gamma_0, \Delta_{\Gamma}))
\rightarrow\pi_1(\Gamma\dbs (\Gamma, \Delta_{\Gamma}))\, .$$ By the
following theorem we also have an injection $$\pi_1(\Gamma_0\dbs (\Gamma_0, \Delta))
\rightarrow\pi_1(\Gamma\dbs (\Gamma, \Delta_{\Gamma}))\,  $$ of the
fundamental group of the quotient graph of groups for $\Gamma_0$ acting
without inversions on $\Delta$ onto the fundamental group of the
quotient h-graph of groups for $\Gamma$ acting with inversions
on $\Delta$.
\begin{theorem}\label{theorem: subdivide?}
Let $\Gamma$ act on a graph $\gr$ without inversions and let
$\gr^\prime$ be obtained from $\gr$ by subdividing all the edges in some
set of edge orbits of $\Gamma$. Then
$$\pi_1(\Gamma\dbs (\Gamma, \gr)) \simeq \pi_1(\Gamma\dbs (\Gamma,
\gr^\prime))\, .$$
\end{theorem}
\begin{proof}
It suffices to consider a single edge in $\Gamma\dbs(\Gamma, \gr)$
\begin{center}
\begin{tikzpicture}
\draw[ultra thick,color = green] (0.3,0)--(5.7,0);
\draw[ultra thick,color = blue] (0, 0) circle [radius = 0.3];
\draw[ultra thick,color = blue] (6, 0) circle [radius = 0.3];
\node at (0,0){$G_0$};
\node at (3.0,0.3){$G$};
\node at (6,0){$G_1$};
\end{tikzpicture}\, \, \, \, .
\end{center} 
If subdivided in $\gr^\prime$ this gives
\begin{center}
\begin{tikzpicture}
\draw[ultra thick,color = green] (0.3,0)--(2.7,0);
\draw[ultra thick,color = green] (3.3,0)--(5.7,0);
\draw[ultra thick,color = blue] (0, 0) circle [radius = 0.3];
\draw[ultra thick,color = blue] (3, 0) circle [radius = 0.3];
\draw[ultra thick,color = blue] (6, 0) circle [radius = 0.3];
\node at (0,0){$G_0$};
\node at (3.0,0.0){$G$};
\node at (1.5,0.3){$G$};
\node at (4.5,0.3){$G$};
\node at (6,0){$G_1$};
\end{tikzpicture}\, \, \, \, ,
\end{center} 
in $\Gamma\dbs(\Gamma, \gr^\prime)$.  The fundamental group of the graph
with the subdivided edge differs from that without only in that
$G_0*_G G_1$ is replaced by $G_0*_G G *_G G_1$, which produces a canonically
isomorphic group.
\end{proof}

To compute amalgamated products for our examples we will need the following two theorems.
\begin{theorem}\label{theorem: free generators}
Suppose $\gr$ has a spanning tree $T$ such that $\Gamma_{\be}$ is
trivial for all $\be \in \Ed(\gr)\setminus \Ed(T)$.  Then $\pi_1(\Gamma,
\gr) \simeq \pi_1(\Gamma, T)*\Z^{*\genus}$, where $\genus = \genus(\gr) =
\#\left(\Ed(\gr) \setminus \Ed(T)\right)$.
\end{theorem}
\begin{proof}
The additional generators $\{x_{\be}: \be \in
\Ed(\gr)\setminus \Ed(T)\}$
are subject only to the trivial relations $x_e = x_e$.
\end{proof}

\begin{theorem}\label{theorem: loop}
  Let $\Gr = (\Gamma, \gr)$ be a graph of groups that consists of a
  single loop such that the stabilizer group of every edge and vertex
  is the same group
$G$ and the induced automorphism of $G$ from the maps around the loop
  is inner.  Then
$\pi_1(\Gr) \simeq G \oplus \Z$.
\end{theorem}
\begin{proof}
Remove one edge $\be$ to form a spanning tree $T$.
Now $\pi_1(\Gr)$ is generated by $\pi_1(\Gamma,T) = G$ and an
additional generator $x_{\be}$ subject to the constraint $g(x_{\be}h) =
x_{\be} hgh^{-1}h=(x_{\be}h)g$ for some $h$ and all $g \in G = \pi_1(\Gamma, T)$.
\end{proof}
It is clear that
$\pi_1(\Gamma, T)$ lies in the kernel of \eqref{pi_proj}. In the case that
the $\Gamma_\bv$ for $ \bv \in \Ver(\gr)$ are finite, then the kernel of
\eqref{pi_proj} is the subgroup $\pi_1(\Gamma, \gr)_{f}$ generated by
all elements of $\pi_1(\Gamma, \gr)$ of finite order.  In particular, if
$\Gamma$ assigns the trivial group to each edge and vertex in $\gr$, then
\eqref{pi_proj} is an isomorphism.

We will use the following to show that $\PGg_n \sinf (\PUT_n)_f$
in some cases.
\begin{prop}\label{prop:inf-index}
Let $(\Gamma, \gr)$ be a connected graph of groups all of whose edge
groups are finite, and let $S$ a subtree of $\gr$. Let $T$ be a
spanning tree of $\gr$ containing $S$ with $\pi_1(\Gamma, T)$
infinite.  Then either $\pi_1(\Gamma, T) = \pi_1(\Gamma, S)$ or else
$\pi_1(\Gamma, S)\ll \pi_1(\Gamma,T)$.  In the second case
$\pi_1(\Gamma, S) \ll \pi_1(\Gamma, \gr)_{f}$.
\end{prop}
\begin{proof}
  If $\pi_1(\Gamma, S)$ is finite the result is trivial, so we will assume
  that it is infinite.  
  Suppose that the natural map $\pi_1(\Gamma, S) \to \pi_1(\Gamma, T)$
  is not surjective. Let $T^\prime$ be the tree obtained from
  collapsing $S$ down to a single vertex ${\bf s}$: then
  $\Ver(T^\prime) = (\Ver(T) \setminus \Ver(S)) \cup \{{\bf s}\}$.
  Edges between a vertex $\bv \in \Ver(T) \setminus \Ver(S)$ and 
  $\bw \in \Ver(S)$ now connect $\bv$ to ${\bf s}$.  We make $T^\prime$
  into a graph of groups $(\Gamma^\prime, T^\prime)$ by defining
  $\Gamma^\prime_{\bf s} = \pi_1(\Gamma, S)$ and $\Gamma^\prime_{\bv}
  = \Gamma_{\bv}$ for $v \in \Ver(T) \setminus \Ver S$; the edge
  groups are the same as they were in $(\Gamma, T)$. We claim that
  $\pi_1(\Gamma^\prime,T^\prime) =\pi_1(\Gamma, T)$. Indeed,
  $\pi_1(\Gamma, T)$ is the amalgam of stabilizers of vertices of $T$
  over stabilizers of edges. As $T$ is a tree, it does not matter in
  which order one amalgamates. Now $\pi(\Gamma^\prime,T^\prime)$ is
  obtained from $(\Gamma, T)$ by first amalgamating over $S$ to give
  $(\Gamma^\prime, T^\prime)$ and then doing the remaining
  amalgamations in $T$.  Since $\Gamma_{\bf s} = \pi_1(\Gamma, S) \ne
  \pi_1(\Gamma, T)$, there must be a vertex $\bv \in
  \Ver(T^\prime)\setminus \{{\bf s}\}$ with a path $p$ between $\bv$ and
  ${\bf s}$ such that all the intermediate vertices and edges along it
  have the same group $\Gamma_p$, but $\Gamma_{\bv}\supsetneq\Gamma_p$.
  We also have $\Gamma_{\bf s}=\pi_1(\Gamma, S)$ infinite, hence bigger
  than the finite group $\Gamma_p$.  Thus,
  $\Gamma_{\bv}*_{\Gamma_p} \Gamma_{\bf s}$ is a nontrivial
  amalgamation.  Therefore, by the normal form for amalgams
  \cite[Theorem IV.2.6]{LS}, we have
$$ \pi_1(\Gamma, S) = \Gamma_{\bf s} \ll \Gamma_{\bv} *_{\Gamma_p} \Gamma_{\bf s} < \pi_1(\Gamma^\prime ,T^\prime) = \pi_1(\Gamma, T) < \pi_1(\Gamma,T)_f.$$ 
\end{proof}

\section{Unitary groups over cyclotomic rings}
\label{sec:cyc-rings-fields}

Our notation will be consistent with that of \cite{IJKLZ2}.
We assume $n=2^s d$ with $d$ odd and $s\geq 2$, $n\geq 8$; put
$\zeta_n=e^{2\pi i/n}$.
Let $K_n=\Q(\zeta_n)$.
The ring of integers in $K_n$ is $\OO_n:=\Z[\zeta_n]$ and
its class group is $\Cl(\OO_n)=\Pic(\OO_n)$ with
class number $h_n=\#\Cl(\OO_n)$. Put $R_n=\Z[\zeta_n, 1/2]$.
 If $H\leq K_n^\times$, put
  \mbox{$H_{1}:=\{x\in H\mid x\overline{x}=1\}$}. Let $F_n=\Q(\zeta_n)^+$
  with integers $\uOOn:=\OO_n^+=\Z[\zeta_n+\overline{\zeta}_n]$,
  class group $\Cl(\uOOn)$ with class number $h(\uOOn)=h_n^+$, and
  narrow class group $\widetilde{\Cl}(\uOO_n)$ with narrow class
  number $\tilde{h}(\uOOn)=\tilde{h}^+_n$.
  Then $h_n=h_n^+ h_n^-$.
  Set $\uRn=R_n^+=\uOO_n[1/2]$.
  For a subgroup $G \leq F_n^\times$, let $G_+$ be the subgroup of $G$ consisting
  of totally positive elements: we have $G/G_+ \cong (\Z/2\Z)^{c_G}$, where
  $0 \le c_G \le [F_n:\Q]$.

\subsection{Cyclotomic Fields}
\label{glass6}

Let $\fp_i$, \mbox{$1 \leq i \leq r_+(n)$}, be the $r_+(n)$ 
prime ideals in $\uOO_n$ above the prime ideal $(2)$ of $\Z$.
If there is a unique prime above $(2)$ in $\uOO_n$, we denote it by
$\fp=\fp(n)$.
Let $\fP_{1},\ldots,\fP_{r(n)}$ be the prime ideals of $K_n$
above $(2)$.   If $\fp_i$ splits in $K_n$, then $2r_+(n)=r(n)$;
if $\fp_i$ is inert or ramified in $K_n$, then $r_+(n)=r(n)$.
If there is a unique prime above $(2)$ in $\OO_n$, we denote it by $\fP=\fP(n)$.

\begin{remark1}  
  \label{pincer}
{\rm
  We have $r_+(n)=r(n)$ if and only if $-1\in\langle 2\rangle\subseteq
  (\Z/d\Z)^{\times}$.
  }
\end{remark1}

We must determine various groups of units.
It is well known that $\Z[\zeta_n]^{\times}\cong\mu_{n}\times
\Z^{\phi(n)/2 -1}$ and that $\Z[\zeta_n]^{\times}_{1} = \mu_{n}$.
Further, $R_n^{\times}$ is generated by
$\Z[\zeta_n]^\times$ and one additional generator for each prime
dividing $2$ in $K_n$; it is thus isomorphic to \mbox{$\mu_{n} \times
\Z^{\phi(n)/2-1+r(n)}$}.  Similarly,
$\uRn^\times$ is isomorphic to $\Z/2 \times
\Z^{\phi(n)/2-1+r_+(n)}$ and
$\uR_{n,+}^\times\cong\Z^{\phi(n)/2-1+r_+(n)}$.
Recall that $r(n)$ is either $r_+(n)$ or $2r_+(n)$.

Now consider $R_{n,1}^\times:=(R_n^\times)_1$.
Let $\Nm=\Norm_{K_n/F_n}$.
There is an exact sequence
$$1\to R_{n,1}^\times \to R_n^\times \xrightarrow{\Nm}
\uR_{n,+}^\times \to G \to 1 ,$$ where $G$ is a
finite group since $\Nm(R_n^\times)\supseteq (\uRn^\times)^2$.
%
%
  Thus
\begin{equation}
\label{onehalfone}
R_{n,1}^\times \cong
\begin{cases}
\mu_{n} & r_+(n) = r(n), \\
\mu_{n} \times \Z^{r_+(n)} & 2r_+(n) = r(n); \\
\end{cases}
\end{equation}
a slightly weaker form of this statement is given
in \cite[Theorem 5.3]{FGKM}.
It follows immediately that 
\begin{equation}
\label{ohosq}
R_{n,1}^\times/(R_{n,1}^\times)^2 \cong
\begin{cases}
\mu_{n}/\mu_{n}^2 & r_+(n) = r(n), \\
\mu_{n}/\mu_{n}^2 \times (\Z/2\Z)^{r_+(n)} & 2r_+(n) = r(n). \\
\end{cases}
\end{equation}
Hence from \eqref{ohosq} we get
\begin{equation}
  \label{ohosq1}
  R_{n,1}^\times/(R_{n,1}^\times)^2 \cong (\Z/2\Z)^{1+r(n)-r_+(n)}.
  \end{equation}

 We are interested in the groups
$\UT(R_n)$ and $\SUT(R_n )$.  The group $\SUT(\Z[\zeta_n])$ is finite:
specifically it is the dihedral group of order $2n$.
But $\SUT(R_n)$ (and {\it a fortiori} $\UT(R_n)$) is infinite.
In fact, by strong approximation at the place $2$ \cite[Main Theorem]{Kn}, 
$\UT(R_n )$ is a dense subgroup of $\UT (\BC )$.

 We have natural inclusions
\begin{equation*}
\SUT(R_n)\hookrightarrow \UTz(R_n)\hookrightarrow \UT(R_n) .
\end{equation*}
 For any complex unitary matrix $A$, the condition $A^{-1}=\overline{A}^{t}$
implies that $\alpha = \det (A)$ satisfies $\alpha\overline{\alpha}=1$.
Hence if $A\in \UT(R_n)$, then $\alpha =\det (A)\in R_{n,1}^\times$.
Also, if
$\alpha\in R_{n,1}^{\times}$, then
$\left[\begin{smallmatrix}1&0\\ 0&\alpha \end{smallmatrix}\right]
\in \UT(R_{n})$.
It follows that there is an exact sequence
\begin{equation}
\label{void}
1\longrightarrow\SUT(R_n)\longrightarrow \UT(R_n)\stackrel{\det}{\longrightarrow}
R_{n,1}^{\times}\longrightarrow 1.
\end{equation}
\begin{prop}\textup{ (\cite[Theorem 5.3]{FGKM})}
\label{pinto}
We have $\UTz(R_n)=\UT(R_n)$ if and only if \mbox{$-1\in\langle 2\rangle \subseteq
(\Z/d\Z)^{\times}$}.
\end{prop}
\begin{proof}
   Combine the exact sequence (\ref{void}) above with Remark \ref{pincer}
  and (\ref{onehalfone}).
\end{proof}

\subsection[$\PUT(R_n),\PUT(R_n),\,
  \PUTz(R_n),\text{ and }
  \PSUT(R_n)$]
{\texorpdfstring{\boldmath{$\PUT(R_n),\,
  \PUTz(R_n),\text{ and }
  \PSUT(R_n)$}}{Unitary groups}}
\label{sec-rat}

We begin our study of unitary groups over cyclotomic rings
by explaining the relationship
between $\PUT(R_n)$, $\PUTz(R_n)$,  and $\PSUT(R_n)$.
  
There is a commutative diagram with exact rows and columns
(to save space we do not indicate the trivial groups on the sides):
\begin{equation}
\label{reds}
  \begin{CD}
   \mu_2 & @>>> & R_{n,1}^\times & @>{(\hspace*{1em})^{2}}>> &
   (R_{n,1}^{\times})^2\\
  @VVV &  & @VVV & & @VVV  \\
\SUT(R_n) & @>>> & \UT(R_n) & @>{\det}>> & R_{n,1}^{\times}\\
  @VVV & & @VVV & & @VVV  \\
\PSUT(R_n) & @>>> & \PUT(R_n) & @>{\det}>> & R_{n,1}^{\times}/
(R_{n,1}^\times)^{2} .
\end{CD}
\end{equation}
 The structure of $R_{n,1}^\times/
(R_{n,1}^\times)^2$ is given in (\ref{ohosq}).
In particular we have
\begin{prop}
  \label{proj}
  \begin{enumerate}[\upshape (a)]
  \item
    \label{aa}
      $\PUT(R_n)/\PSUT(R_n)\isisom (\Z/2\Z)^{1+r(n)-r_+(n)}$.
  \item
    \label{bb}
        $\PUT(R_n)/\PUTz(R_n)\isisom (\Z/2\Z)^{r(n)-r_+(n)}$.
  \item
    \label{cc}
 $\PUTz(R_n)/\PSUT(R_n)\isisom \Z/2\Z.$
\end{enumerate}
  \end{prop}
\begin{proof}
  By diagram (\ref{reds}) $\PUT(R_n)/\PSUT(R_n)\isisom R_{n,1}^\times/
  (R_{n,1}^\times)^{2}$, hence \eqref{aa} follows from (\ref{ohosq1}).

  Similarly, $\PUTz(R_n)/\PSUT(R_n)\isisom \mu_{n}/\mu_{n}^2\isisom
  \Z/2\Z$ since by diagram (\ref{reds}) the determinant map is
  surjective.  The claim \eqref{cc} follows.

Assertion \eqref{bb} now follows trivially.
\end{proof}

If $r(n)_+=r(n)$, i.e., if primes above $2$ in $F_n$
do not split in $K_n$, then the commutative diagram \eqref{reds}
becomes
\begin{equation}
  \label{reds1}
\begin{CD}
\mu_2 & @>>> & \mu_{n}& @>{(\hspace*{1em})^{2}}>> & \mu_{n}^2\\
  @VVV &  & @VVV & & @VVV  \\
\SUT(R_n) & @>>> & \UT(R_n) & @>{\det}>> & \mu_{n}\\
  @VVV & & @VVV & & @VVV  \\
\PSUT(R_n) & @>>> & \PUT(R_n) & @>{\det}>> & \mu_{n}/
\mu_{n}^{2} .
\end{CD}
\end{equation}

Proposition \ref{punt} below is elementary.
  \begin{prop}
    \label{punt}
    The following are equivalent:
    \begin{enumerate}[\upshape (a)]
    \item
      There is a unique prime $\fp$ of $F_n$ above $2$, i.e., $r_+(n)=1$.
      \item
    We have
    $\langle 2,\, -1\rangle = (\Z/d\Z)^{\times}$.
\end{enumerate}
    \end{prop}

\begin{prop}
  \label{hearts}
  The following are equivalent:
  \begin{enumerate}[\upshape (a)]
  \item
    \label{cow}
      $r(n)=r_+(n)$.
  \item
    \label{calf}
        $-1\in\langle 2\rangle\subseteq \left(\Z/d\Z\right)^{\times}$.
  \item
    \label{sheep}
          $\PUT(R_n)/\PSUT(R_n)\cong \Z/2\Z$.
  \item
    \label{duck}
          $\PUT(R_n)=\PUTz(R_n)$.
          \end{enumerate}
\end{prop}
\begin{proof}
The equivalence of \eqref{cow} and \eqref{calf} is elementary.
The equivalence of \eqref{cow} and \eqref{sheep}
follows from diagram \eqref{reds}
and \eqref{ohosq}.  The equivalence of \eqref{sheep} and \eqref{duck} follows from
Proposition \ref{proj}(c).

\end{proof}

\subsection[The Clifford-cyclotomic groups
  $\Gg_n$ and $\SGg_n$]
    {\texorpdfstring{The Clifford-cyclotomic groups
  \protect{\boldmath{$\Gg_n \text{ and }\SGg_n$}}}
{The Clifford-cyclotomic groups}}
\label{report}

The {\em Clifford group} $\mathcal{C}$ can be defined as
$\mathcal{C}=\UT(R_4)$ \cite[Section 2.1]{FGKM}.  With
$T_n$ as in \eqref{guestroom}, 
define the {\em Clifford-cyclotomic group} \cite[Section 2.2]{FGKM}(resp.,
{\em special} Clifford-cyclotomic group) for $4|n$  by
\begin{equation}
  \label{denied2}
  \Gg_n=\langle\mathcal{C}, T_n\rangle \quad\quad
  \mbox{(resp., $\SGg_n=\Gg_n\cap \SUT(R_n)$);}
\end{equation}
we have $\Gg_n\subseteq \UTz(R_n)$. This definition agrees with
\eqref{tooth} by \cite[Prop. 2.1]{IJKLZ}.
For additional results on $\Gg_n$ and $\SGg_n$ see
\cite{IJKLZ}.

\begin{prop}\label{cor:other-inf-index}
  Suppose that $\PGg_n \ll [\PUT(R_n)]_f$.
  Then \mbox{$\Gg_n \ll \UT(R_n)_f$},
  $\SGg_n\ll\SUT(R_n)_f$, $\Pro\!\SGg_n\ll[\Pro\!\SUT(R_n)]_{f}$.
\end{prop}

\begin{proof}
  The subgroup of scalar matrices of $\Gg_n$ and the image of the determinant
  homomorphism $\Gg_n \to \BC$ are always finite; likewise for $\UT(R_n)_f$.
\end{proof}

\section{The Hamilton quaternions and unitary groups}
 \label{glass8}
\subsection{The Hamilton quaternions}
 \label{sec:Hamilton quaternions}
   
Let $\H$ be the Hamilton quaternions over $\Q$ with a fixed 
$\Q$-basis $1$, $i$, $j$, $k$
satisfying $i^2=j^2=k^2=-1$, $ij=-ji$, $ik=-ki$, $jk=-kj$.
Put $\H_n=\H\otimes_{\Q} F_n$. 

  \begin{prop}
    \label{unramified}
    Let $n=2^s d$ with $d$ odd, $s\geq 2$, and $n\geq 8$.
    Then the quaternion algebra $\H_n$ is
    unramified at the primes above $2$ in
    $F_n$.  Equivalently, $\H_n$ is
    unramified at all finite primes of $F_n$.
  \end{prop}
\begin{proof} 
  The quaternion algebra $\H_n$ is unramified at $\fp_i$ for
   $1\leq i\leq r_+$ if and only if the order of the decomposition group
   $$D(\fp_i)
   \subseteq \Gal(F_n/\Q)\simeq (\Z/n\Z)^{\times}/\langle \pm 1\rangle$$
   is even.  If $s > 2$, then $F_n$ contains $\Q(\zeta_{2^s})^+$, in which
   $e(\fp)$ is even.  For $s = 2$, the extension $\Q(\zeta_n)$ has
   ramification index $2$ above $2$, with the inertia field being
   $\Q(\zeta_d) \not \supseteq F_n$ since
   $n\geq 8$. Thus $F_n$ likewise
   has ramification index $2$ above $2$, so the decomposition
   group has even order as well.
\end{proof}
The following assertion is elementary.
\begin{prop}
  \label{one/unramified}
  Let $n=2^s d$ with $d$ odd, $s\geq 2$, and $n\geq 8$.
 The following are equivalent:\begin{enumerate}[\upshape (a)]
  \item
   There is a unique prime
$\fp$ of $F=F_n$ above $2$ and
  the quaternion algebra $\H_n$
  is unramified at that prime $\fp$.
  \item
  Hypothesis \textup{\ref{assume}:}
  $\langle 2, -1\rangle = (\Z/d\Z)^{\times}$.
  \end{enumerate}
\end{prop}

 \begin{prop}
  \label{punt1}
Let $n=2^sd$ with $d$ odd, $s\geq 2$, and $n\geq 8$.
The following are equivalent:
\begin{enumerate}[\upshape (a)]
  \item
  There is a unique prime $\fp$ of $F_n$ above $2$,
  the quaternion algebra $\H_n$ is unramified at $\fp$, and $\fp$
  does not split in $K_n$.
  \item
 $\langle 2\rangle = (\Z/d\Z)^{\times}$.
\end{enumerate}
\end{prop}
 \begin{proof}
   Combine Proposition \ref{one/unramified} with Proposition \ref{hearts}.
   \end{proof}
   
 \noindent  Obviously $n = 2^s$ satisfies the conditions in
 Proposition \ref{punt1} 
for $s\geq 3$.  However, $n = 8m$ also satisfies these conditions
for $m \in \{3,5,6\}$, although not for $m = 7$.
We will examine many of these
graphs for $n$ a small multiple of $4$ explicitly in this paper. 

\begin{prop} 
   \label{spades}
   Assume $n=2^s d$ with $d$ odd, $s\geq 2$,
   and $n\geq 8$ with $\langle 2,-1\rangle =(\Z/d\Z)^\times$.
   The following are equivalent:
   \begin{enumerate}[\upshape (a)]
   \item
     \label{rose}
       $H^{1}(\Gal(K_n/F_n), R_n^{\times})=0$.
   \item
     \label{rose2}
  If $\fp$ splits as $\fp=\wp\overline{\wp}$ in $K_n$ and $r$
     is the least positive integer such that $(\wp/\overline{\wp})^r =
     (\beta )$ is principal with $\Nm_{K_n/F_n}(\beta)=1$, then
     \mbox{$(1+\beta)\delta\in R_n^\times$} for some $\delta\in F_n$.
   \end{enumerate}
 \end{prop}
\begin{proof}
  Note that \eqref{rose} is equivalent to the statement that every
  $\alpha\in R_{n,1}^\times$ is given by $\gamma/\overline{\gamma}$
  for some $\gamma\in R_n^\times$.  By
  \eqref{onehalfone}, $R_{n,1}^\times\cong\mu_{n}$ if $\fp$ does not
  split in $K_n$ and $R_{n,1}^\times\cong\mu_{n} \times \Z$ if it
  does, where $\mu_{n}$ is generated by $\zeta_n$ and $\beta$ is a
  generator of the $\Z$ since it has norm $1$ and is the ``smallest''
  generator that does so.

  Assume \eqref{rose2}. By \cite[Lemma 3.9]{IJKLZ2}
  $1+\zeta_n\in R_n^\times$.  Thus
  $\zeta_n=\gamma/\overline{\gamma}$ for $\gamma=1+\zeta_n\in
  R_n^\times$ and $\beta=\gamma/\overline{\gamma}$ for
  $\gamma=(1+\beta)\delta\in R_n^\times$.  Hence, since $R_n^{\times}$
  is generated by $\zeta_n$ and $\gamma$, we have
  $H^{1}(\Gal(K_n/F_n), R_n^{\times})=0$ and \eqref{rose} is true.

  Conversely, assume \eqref{rose}.  Then we have
  $\beta=\gamma/\overline{\gamma}$ for some $\gamma\in R_n^\times$.
  Let $\delta=\gamma/(1+\beta)$.  Then
  $$\overline{\delta}=\frac{\overline{\gamma}}{1+\overline{\beta}}
  =\frac{\gamma\beta}{1+\overline{\beta}}=\frac{\gamma}{1+\beta}=\delta .$$
  Therefore, $\delta\in F_n$ and \eqref{rose2} follows.
\end{proof}

 The {\it standard} maximal $\uRn$-order of $\H_n$ is 
\begin{equation*}
  \label{standard}
  \widetilde{\M}_n := \uRn\langle 1,\,i,\,j,\, (1+i+j+k)/2 \rangle\  .
\end{equation*}
Now for each $n$ we choose an $\uOO_n$-maximal order
$$\M_n\supseteq \{1,\,i,\,j,\, (1+i+j+k)/2 \}.$$
 The ideal $(2)$ in $\uOO_n$ is the square of an ideal 
  $\fq=\fq_n$.
Fix a set of generators $A=A(n)$ for $\fq$. For example, if $8|n$ we take
$A = \{\sqrt{2}\}$; if not but $12|n$, then take
$A = \{1+\sqrt{3}\}$.
Define the maximal $\uOO_n$-order $\M_n\subseteq\widetilde{\M}_n$ by
\begin{equation}
  \label{eq: Mn}
\M_n=\uOO_n\langle 1, (1+i)\alpha/2, (1+j)\alpha/2,(1+i+j+k)/2\rangle\, ,
\end{equation}
where $\alpha$ runs over $A$.  Observe that $\M_n$ does not depend
on the choice of generators $A=A(n)$ of $\fq=\fq_n$.

 \begin{remark1}
  In general, $\uOO_n\langle 1,\,i,\,j,\,(1+i+j+k)/2 \rangle$ is not a maximal
  order of $\H_n$.  Indeed, this order has discriminant $(2)$; if $\H_n$
  is unramified at the primes above $2$ (for example, if $n=2^s$), then
 the discriminant of a maximal order
  of $\H_n$ is the unit ideal.  On the other hand, the order $\widetilde{\M}_n$
  is a maximal $\uRn$-order, because $2$ is a unit in $\uRn$.
\end{remark1}

  \begin{remark1}  
    {\rm In general the $\uRn$-type number of $\H_n$
      is not $1$---there can be nonisomorphic $\uRn$-maximal orders
      of $\H_n$.}
  \end{remark1}

  We now make definitions as in Kurihara \cite{K}(who in turn follows
  Ihara \cite{I}):
  \begin{definition1}
    \label{carrot}
    Assume $n$ satisfies Hypothesis \ref{assume} with $\fp$ the unique
    prime of $F:=F_n$ above $2$. Note that $\H_n\otimes_{F}F_{\fp}=
    \Mat_{2\times 2}(F_\fp)$.
    Set
    \begin{align*}
      \widetilde{\M}_{n,1}^\times &=\{m\in\widetilde{\M}_n^\times\mid
      \Norm_{\H_n/F_n}(m)=1\}\\
      \widetilde{\M}_{n,+}^\times &=\{m\in\widetilde{\M}_n^\times\mid
    \val_{\fp} (\Norm_{\H_n/F_n}(m))\text{ is even}\}.
\end{align*}
Define:
    \begin{align*}
      \Gamma_0 =\Gamma_{0,n}& =\Gamma_{0,n}(\widetilde{\M}_n)
      =\Pp\!\widetilde{\M}_n^\times =\widetilde{\M}_n/\uRn^\times\\
      \Gamma_+=\Gamma_{+,n}&=\Gamma_{+,n}(\widetilde{\M}_n)=
      \Pp \!\widetilde{\M}_{n,+}^\times =\widetilde{\M}_{n,+}^\times/\uRn^\times\\
      \Gamma_1=\Gamma_{1,n}&=\Gamma_{1,n}(\widetilde{\M}_n)=
      \Pp \!\widetilde{\M}_{n,1}^\times =\widetilde{\M}_{n,1}^\times/\pm 1.
    \end{align*}
    Then $\Gamma_1\subseteq\Gamma_{+}\subseteq\Gamma_{0}$ are discrete,
    cocompact subgroups of $\PGLT(F_\fp)$.
    \end{definition1}

Recall that if $H \leq F_n^\times$, then $H_+$ is the subgroup of totally
positive elements of $H$. 
Assume that $n$ satisfies Hypothesis \ref{assume} with $\fp$ the unique
prime of $F:=F_n$ above $2$.
Put
\[
\uR^\times_{n,+,\,\fp\text{-ev}}=\{x\in\uR_{n,+}^\times\mid
\val_{\fp}(x)\text{ is even}\}.
\]
 The reduced norm map 
  $\Nm=\Norm_{\H_n/F_n}\colon \H_n^{\times}\rightarrow F_n^{\times}$ induces maps
 \begin{equation}
   \label{rooster}
     \Nm_0:\Gamma_0\rightarrow\frac{\uR_{n,+}^\times}
     {(\uRn^{\times})^{2}}\,\,,\quad
     \Nm_{+}\colon \Gamma_{+}\rightarrow\frac{\uR^\times_{n, +,\,\fp\text{-ev}}            }
     {(\uRn^\times)^2}\,\, ,\quad
     \Nm_1\colon \Gamma_1\rightarrow 1.
 \end{equation}

  Let $C_2$ be the cyclic group of order $2$, which we identify both with
      $\pm 1$ and with $\F_2$.  For $1 \le i \le d=[F:\Q]$,
      let $s_i$ be the map $F^\times \to C_2$
      taking $x$ to the sign of its image in the $i$-th real embedding of $F$.
      We then define the \mbox{$\fp$-signature} map
      $\sig_\fp \colon  \uRn^\times\rightarrow {C_2}^{d+1}$ 
      by
      \begin{equation}
        \label{sig}
        \sig_\fp(x) = (s_1(x),\dots,s_{d}(x),\val_\fp(x) \bmod 2) .
        \end{equation}
      
 \begin{prop}
   \label{house}
   \begin{enumerate}[\upshape (a)]
   \item
     \label{house1}
  The maps $\Nm_0, \Nm_{+}, \Nm_{1}$
  in \textup{(\ref{rooster})} are surjective.
\item
  \label{house2}
                There are isomorphisms
     \begin{equation*}
              \Gamma_{0}/\Gamma_{+}\cong\frac{\uR^\times_{n,+}}
                    {\uR^\times_{n,+,\, \fp\text{\rm{-ev}}}}\quad\mbox{and}
                    \quad
                    \Gamma_{+}/
                    \Gamma_{1}\cong\frac{\uR^\times_{n, +,\, \fp\text{\rm{-ev}}}}
   {(\uRn^\times)^2}.                    
     \end{equation*}
   \item
     \label{house3}
                $\#\Gamma_{0}/\Gamma_{+}\leq 2$,
     with equality if and only
     if the class $[\fp]$ of $\fp$ in $\Cln (F)$ of $F$ is of odd order.
   \item
     \label{house4}
   We have $\Gamma_{+}/\Gamma_{1}\cong
   \Coker(\sig_{\fp}) \cong \F_{2}^r$ with $0\leq r\leq d=[F:\Q]$.
   \end{enumerate}
 \end{prop}
 \begin{proof} First we show that $\Nm_0$ is surjective: pick any
   $x\in\uR_{n,+}^\times$.  Then by \cite[Lemma 3.19]{IJKLZ2}
   there
   exists a $\gamma\in\widetilde{\M}$ of norm $x$.  Observe that $\gamma$ is
   a unit since its norm is.  Thus $\gamma$ gives an element of
   $\Gamma_0$ and $\Nm_0$ is surjective.  A similar argument holds for
   $\Nm_+$ and $\Nm_1$.

   To derive \eqref{house2} from \eqref{house1}, 
   note that all the definitions of the $\Gamma$'s are
   equivalent to the pullbacks under the reduced norm map of the
   groups in (\ref{rooster}).  Thus their quotients are the same as
   the quotients of the images of their norms.

   With that done, \eqref{house3} follows from the second isomorphism in
   \eqref{house2}.  It is
   clear that $\#\Gamma_0/\Gamma_+ \le 2$, and
   the class $[\fp]$ of $\fp$ in $\Cln(F)$ is of odd order if and only if there
   is a totally positive element of $F$ generating the ideal $\fp^k$ for some
   odd $k$.  If there is no such element, then
   $\uR_{n,+}^\times = \uR_{n, +,\, \fp\text{\rm{-ev}}}^\times$
   and the index is $1$,
   whereas if there is such an element it generates the quotient and the index
   must be $2$.

   For \eqref{house4}, note that
    \[
    \frac{\uRn^\times}
         {(\uRn^\times)^2}\cong \F_{2}^{d+1}.
         \]
         The assertion then follows from \eqref{house2} and  the exact sequence 
   \begin{equation*}
 \label{none1}
 1\longrightarrow \frac{\uR_{n,+,\, \fp\text{-ev}}}
 {(\uR_{n}^\times)^2}         \longrightarrow
 \frac{\uRn^\times}
  {(\uRn^\times)^{2}}\stackrel{\sig_\fp}{\,\longrightarrow} \F_2^{d+1}\longrightarrow
  \Coker(\sig_\fp)\longrightarrow 1
   \end{equation*}
   upon observing that $\dim_{\F_2}\Coker(\sig_{\fp})\leq d$
   since $\sig_{\fp}(-1)$ is nontrivial.
\end{proof}
  
  \begin{theorem}
    \label{pills}
    \begin{enumerate}[\upshape (a)]
    \item
      \label{pills1}

      The groups $\Gamma_{0},\, \Gamma_{+},\, \Gamma_1$
      are discrete cocompact
            subgroups of $\PGLT(F_{\fp})$.  Let $\Delta=\Delta_\fp $.
            Then $
            \Gamma_{+},\, \Gamma_1$ act on
            $\Delta$ without inversions and the quotients
            $
            \Grp=\Gamma_{+}\backslash \Delta,\,
            \Gro=\Gamma_{1}\backslash \Delta$ are finite bipartite graphs.
            The group
            $\Gamma_{0}$ acts on
            $\Delta$ possibly with inversions; the quotient
            $\Grz=\Gamma_{0}\backslash\Delta$ is a  finite Kurihara
            graph.
          \item
            \label{pills2}
          The natural covering $\pi\colon  \Grp\rightarrow\Grz$
    is \'{e}tale of degree $1$ or $2$. The degree is $2$ if and only if
    the class $[\fp]$ in $\Cln (F)$ is of odd order.
    \end{enumerate}
  \end{theorem}

  \begin{proof}
    The assertion \eqref{pills1} follows from \eqref{house3} and
\eqref{house4}
of Proposition \ref{house}.  A general discussion is in \cite[Section 4]
{K}.

Part \eqref{pills2} follows from Proposition \ref{house}\eqref{house3}.
  \end{proof}

  \subsection{Connecting unitary groups to the Hamilton quaternions}
\label{fiend}

Let $\H_{n,1}^\times$ be the subgroup of
$\H_n^\times$ of elements of norm $1$.
The following observation is standard and easy to check:

\begin{prop}\label{suquat}  For all $n$,
the map $\SUT (\cyc{n}) \to \H_{n,1}^\times$ defined by
$$\pmat{r+s\sqrt{-1}&t+u\sqrt{-1}\\-t+u\sqrt{-1}&r-s\sqrt{-1}} \mapsto
r - ui - tj - sk$$
is an isomorphism. 
\end{prop}

 The map in Proposition \ref{suquat} restricts to an isomorphism
\begin{equation*}
  \Psi_n\colon \SUT(R_n)\stackrel{\simeq}{\longrightarrow}\widetilde{\M}_{n,1}^{\times},    
\end{equation*}
with an induced isomorphism
\begin{equation}
    \label{rounder1}
\overline{\Psi}_{n}\colon \PSUT(R_n)=\SUT(R_n)/\langle\pm 1\rangle\stackrel{\simeq}{
    \longrightarrow} \Pro\!\widetilde{\M}_{n,1}^{\times}:=
    \widetilde{\M}_{n,1}^{\times}/\langle \pm 1\rangle.
\end{equation}
We now ask whether there is an isomorphism for $\PUT$ compatible with
the isomorphism (\ref{rounder1}) for $\PSUT$.

First we define a map $\varphi_n\colon
\PUT(K_n)\mapsto\Pro\! \H_n^\times$.

\begin{definition1}
  \label{phi}
  {\rm
  For $A\in \UT(K_n)$, denote by $[A]$ its class in $\PUT(K_n)$.  Similarly,
  for $a\in \H_n^{\times}$, denote by $[a]$ its class in
  $\Pro\! \H_n^{\times}$.
  
  Suppose $A\in\UT(K_n)$ and $\alpha=\det(A)$ where
  $\alpha\overline{\alpha}=1$.  By Hilbert's Theorem 90 there is
  $\beta \in K_n^\times$ such that $\alpha = \overline{\beta}/\beta$.  Let
  $A'=\beta A$.  We have $\det A'=\beta^2\alpha =
  \beta\overline{\beta}\in F_n$. Hence $A'$ is of the form
  \begin{equation*}
    A'=\left(\begin{array}{rr}
        r+s\sqrt{-1} & t + u\sqrt{-1}\\
        -t+u\sqrt{-1} & r-s\sqrt{-1}
      \end{array}\right)
  \end{equation*}
  and we then define, for $[A]\in \PUT(R_n)$,
  \begin{equation}
    \label{pasta}
    \ovn([A])=[r-ui-tj-sk]\in\Pro\! \H_n^\times.
  \end{equation}
  }
\end{definition1}
\noindent Note that on $\PSUT(R_n)$ our map $\ovn$ agrees with
$\overline{\Psi}_{n}$.
\begin{remark1}
  \label{floors}
  Under the equivalent conditions of Proposition {\rm \ref{one/unramified}},
  the map $\ovn$ makes $\PUT(R_n)$ a discrete subgroup of $\PGL_{2}(F_\p)$,
  since $\PSUT(R_n)$ has finite index in $\PUT(R_n)$.
\end{remark1}
Applying $\ovn$ to $H$ and $T_n$ given in \eqref{guestroom} using $\beta = 1+\sqrt{-1},
1+\zeta_n^{-1}$, respectively, we obtain
$$\ovn([H]) = [h], \quad \ovn([T_n]) = [t_n], \,$$
where 

\begin{equation}\label{eq: H, Tn}
h := -i-k, \quad t_n =1+e^{2\pi k/n}:= 1+\frac{\zeta_n+\zeta_n^{-1}}{2} - \frac{(\zeta_n-\zeta_n^{-1})\sqrt{-1}k}{2}.
\end{equation}
\begin{remark1}\label{rem: h}
 Clearly, $h \in \M_n$ when $\M_n$ is of the form \eqref{eq: Mn}. 
\end{remark1}
\begin{theorem}
  \label{thm: tn}
Recall that $4|n$ and $n > 4$.  Let $\T_n$ be the $\uOOn$-order in $\H_n$ 
generated by $t_n$ and $j$.
\begin{enumerate}[\upshape (a)]
\item
  \label{thm: tn1}
If $n$ is not a power of $2$, 
then $\T_n$ is maximal and $\T_n^\times/\uOOn^\times \simeq D_n$, the dihedral
group of order $2n$.
\item
\label{thm: tn2}
  If $n$ is a power of $2$, then $\T_n$ is not maximal.  Its discriminant
is $\fp^2$ and it is contained in exactly two maximal orders, the order 
generated by $t_n$ and $(1+i)/(\zeta_n + \zeta_n^{-1})$ and its conjugate by
$t_n$.  The intersection of these two orders contains $\T_n$ with
quotient $\fp$, and conjugation by $t_n$ exchanges them.  In particular 
$\T_n$ is not an Eichler order.  Further, we have
$\T_n^\times/  \uOOn^\times\simeq D_{n/2}$, the dihedral group of order
$n$, unless $n = 8$, and likewise for $\M$ a maximal order containing $\T_n$ 
we have $\M^\times/\uOOn^\times \simeq D_{n/2}$.  Additionally,  both
$\M \cap \M^{t_n}$ and $\T_n$ have stabilizer $D_n$.
\end{enumerate}
\end{theorem}
\begin{proof}
For this proof only, let $z = \zeta_n + \zeta_n^{-1}$.
Consider the $\uOOn$-submodule $S_n$ generated by 
$1, t_n, j, t_n j$.  Our first claim is that $S_n$ is an order of discriminant
$(z^2-4)$.  To show that it is an order, we check that
it is closed under multiplication.  Indeed, $t_n, j$ are integral, so their
squares are still in the order, and once we show that $jt_n \in S_n$
the remaining products will follow by associativity.  In fact
\begin{equation}
  \label{dill}
  jt_n = (z + 2) j - t_n j.
\end{equation}
  To evaluate $\Disc S_n$, 
we compute the matrix of traces of products of the basis vectors, obtaining
$$\begin{bmatrix}2&z+2&0&0\\z+2&z^2+2z&0&0\\0&0&-2&-z-2\\0&0&-z-2&-2z-4\\
\end{bmatrix}.$$
The top left $2 \times 2$ block has determinant $z^2 - 4$, the bottom right 
$-(z^2-4)$, so our claim follows.

This proves that $S_n = \T_n$ is a maximal order in the case where $n$ is
not a power of $2$, since in that case $z^2 - 4 = (\zeta_n^2-1)^2/\zeta_n^2$ is
a unit.  When $n$ is a power of $2$, the determinant generates the $4$th power 
of the prime $\fp$ of $\uOOn$ above $2$, so the discriminant is
$\fp^2$ and an order that contains $S_n$ 
with quotient $(\mathcal{O}_n/\fp)^2$ is maximal.  
We may enlarge the order by adjoining any of $1+j, (t_n-1)(1+j), t_n(1+j)$
divided by $z$, or their conjugates by $t_n$.  It is
easily checked that all of these orders contain $t_n(1+j)/z$, and that
$t_n(1+j)/z$ generates an Eichler order of discriminant $\fp$ as described.

To prove that the unit group is as claimed, we notice that by construction
$t_n$ is the image of the matrix $T_n$, whose order in $\PGLT$ is exactly $n$.
It also holds that $jt_nj^{-1} = t_n^{-1}$ up to scalars, so in the first case
where $t_n$ is a unit we obtain the
dihedral group $D_n$ as a subgroup of $\T_n^\times/\uOOn^\times$.  In the second
case $t_n$ is not a unit; however, $t_n^2/z$ is a unit, and this gives a
unit group of $D_{n/2}$ as claimed.
To show that $\M^\times/\uOOn^\times$ is no larger than this, we consider the 
list of maximal subgroups of $\PGLT({\bf C})$.  The tetrahedral, octahedral,
and icosahedral groups have no dihedral subgroups larger than $D_5$, so these 
are excluded except in the case $n = 8$, in which we find $\Sc_4$; 
the only other possibility is that the unit group is $D_{kn}$ for some $k > 1$.
In the first case, that is not possible, because the subring of $\H_n$ 
obtained by adjoining an element of order $kn$ to $F_n$ would be a subfield of 
degree $\phi(kn)/[F_n:\Q] > 2$, a contradiction.  In the second case it also
does not occur.  The same argument shows that we would have to have $k = 2$,
and the only possible unit of order $n$ would be a scalar multiple of 
$t_n$.  However, no such multiple has unit norm, so we cannot obtain an
automorphism of a maximal order this way.  On the other hand, conjugation
by $t_n$ is of order $2$, since $t_n^2$ is a scalar.  Thus the
Eichler order $\M \cap \M^{t_n}$ and its canonically defined suborder $\T_n$
are preserved by conjugation by $t_n$; this has order $n$ and, together with
the $D_{n/2}$, generates a group isomorphic to $D_n$.  The same argument as
in case $1$ shows that this is the stabilizer of $\M \cap \M^{t_n}$ and $\T_n$.
\end{proof}
\begin{theorem} 
  \label{candy}
  Assume Hypothesis \textup{\ref{assume}}  and use the notation 
  of Definition \textup{\ref{carrot}.}
  For the map $\ovn$ in \eqref{pasta} 
  \begin{equation*}
    \ovn(\PUT(R_n))\subset
    \Pro\!\widetilde{\M}_n^\times = \Gamma_{0}(\M_n,
    \fp):=\Gamma_{0,n}
  \end{equation*}
    if and only if the equivalent conditions of Proposition  \textup{\ref{spades}}
    are satisfied.
\end{theorem}
\begin{proof}
Using the notation above, notice that $H^{1}(\Gal(K_n/F_n),
R_n^{\times})=0$ implies that $\beta\in R_n^{\times}$.  Hence we end up
with $\ovn([A])\in \Pro\!\widetilde{\M}_n^\times$.

Conversely, suppose
$\ovn(\PUT(R_n))\subset\Pro\!\widetilde{\M}_n^\times$.  Pick any
$\alpha\in R_n^\times$ with $\alpha\overline{\alpha}=1$.  Let
$A\in\UT(R_n)$ have $\det(A)=\alpha$.  Then
$\varphi_n([A])=[r-ui-tj-sk]\in\Pro\!\widetilde{\M}_n^\times$; set
$w:=r-ui-tj-sk\in\widetilde{\M}_n^\times$. Now let
  \begin{equation*}
    A'=\left(\begin{array}{rr}
        r+s\sqrt{-1} & t + u\sqrt{-1}\\
        -t+u\sqrt{-1} & r-s\sqrt{-1}
      \end{array}\right)\in \SUT(R_n).
  \end{equation*}
By  the construction of $\varphi_n$ we have $A'=\beta A$ for some
$\beta\in F_n^\times$ and by the above inclusion we must have
$\beta\in R_n^\times$.  Comparing determinants we get $\beta^2
\alpha=\det(A')=N(w)$.  Taking norms down to $\uRn$ gives
$(\beta\overline{\beta})^2=N(w)^2$.  Hence $\beta\overline{\beta}=N(w)$
since both quantities are totally positive.  Therefore, $\alpha =
N(w)/\beta^2=\overline{\beta}/\beta$ with $\beta\in R_n^\times$.  
Since $[K_n:F_n]  = 2$, this shows that $H^1(\Gal(K_n/F_n,R_n^\times) = 0$,
which implies the
  equivalent conditions of Proposition
\textup{\ref{spades}}.
\end{proof}
 We now ask when the map
$\ovn\colon \PUT(R_n)\rightarrow \Pro\!\widetilde{\M}_n^\times=\Gamma_0(\widetilde{\M}_n)$
is an isomorphism.
\begin{theorem}
  \label{rate}
   Suppose the equivalent conditions of Proposition
   \textup{\ref{spades}} are satisfied; retain the notation of
   Theorem~\textup{\ref{candy}}.  For $\alpha\in
   \uR_{n,+}^\times$, denote by $[\alpha]$ its class in
   $\,\uR^\times_{n,+}/(\uRn^{\times})^2$.
   \begin{enumerate}[\upshape (a)]
   \item
     \label{rate1}
  If $\fp$ does not split in $K_n$, then $\ovn$
  is an isomorphism onto $\Pro\!\widetilde{\M}_n^\times=\Gamma_0(\widetilde{\M}_n)$ if and only if the map below with a \textup{?} is an isomorphism:
\[
\Gamma_{0,n}/\Gamma_{1,n}=
    \langle [(1+\zeta_n)(1+\overline{\zeta}_n)]\rangle \simeq 
    \frac{\uR_{n, +}^\times}{(\uRn^\times)^2}\stackrel{?}{\simeq} \Z/2\Z  \simeq  
  \frac{\PUT(R_n)}{\PSUT(R_n)}\stackrel{\simeq}{\longrightarrow}
 \frac{R_{n,1}^\times}{(R_{n,1}^{\times})^{2}} . 
\]
\item
  \label{rate2}
 If $\fp$ splits in $K_n$, then $\ovn$
   is an isomorphism onto
   $\Pro\!\widetilde{\M}_n^\times=\Gamma_0(\widetilde{\M}_n)$ if and
   only if the map below with a \textup{?} is an isomorphism:
\[
\Gamma_{0,n}/\Gamma_{1,n}  \simeq 
     \frac{\uR_{n,+}^\times}{(\uRn^\times)^2}\stackrel{?}{\simeq} \Z/2\Z\times \Z/2\Z
    \simeq  
 \frac{\PUT(R_n)}{\PSUT(R_n)}\stackrel{\simeq}{\longrightarrow}
 \frac{R_{n,1}^\times}{(R_{n,1}^{\times})^{2}} . 
\]
       In this case note that
       \begin{equation*}
         \langle[(1+\zeta_n)(1+\overline{\zeta}_n)]\rangle\simeq\Z/2\Z  \subseteq 
         \frac{\uR_{n, +}^\times}{(\uRn^\times)^2}\simeq
         \Z/2\Z\times \Z/2\Z.   
         \end{equation*}
\end{enumerate}
\end{theorem}
\begin{proof}
   From the diagram (which omits $1$'s on the right/left and top/bottom)
   with Norm map $\Nm_{0,n}\colon \Gamma_{0,n}\rightarrow \uR_{n,+}^\times$
   
{\small
  $$\begin{CD}
    \PSUT(R_n) & @>{\simeq}>> & \Gamma_{1,n} &
    @>>> & 1\\
  @VVV &  & @VVV & & @VVV  \\
  \PUT(R_n) & @>{\phi_n}>> & \Gamma_{0,n}& @> {\Nm_{0,n}}>> &
 $coker$\\
  @VVV & & @VVV & & @VVV  \\
  \PUT(R_n)/\PSUT(R_n)\simeq R_{n,1}^\times/
  (R_{n,1}^{\times})^2 & @>>> & \Gamma_{0,n}/\Gamma_{1,n}\simeq
  \uR_{n,+}^{\times}/(\uRn^{\times})^2& @>{\Nm_{0,n}}>>
  & $coker$ \, ,
  \end{CD}$$}

  \noindent we see that $\ovn$ is an isomorphism if and only if
  \begin{equation}
    \label{ream}
    \frac{R_{n,1}^{\times}}{(R_{n,1}^{\times})^2} \stackrel{\simeq}
         {\longrightarrow} \frac{\uR_{n, +}^\times}{(\uRn^{\times})^2}
  \end{equation}
  is an isomorphism.
  The map in (\ref{ream}) is induced by $[\zeta_n]\mapsto
  [(1+\zeta_n)(1+\overline{\zeta}_n)]$.
\end{proof}

\begin{remark1}
Note in particular that by Weber's theorem \cite{We} the hypotheses of
Theorem \ref{rate} are always satisfied when $n$ is a power of $2$.
\end{remark1}

We now look at the group $\PU_2^\zeta(R_n)$ defined in the introduction.

\begin{theorem}
  Assume that $n$ is not twice a prime power and use the notation 
 of Definition \textup{\ref{carrot}}. Then
  \begin{equation*}
     \ovn(\PU_2^\zeta(R_n))\subset 
    \Gamma_{+, n}(\widetilde{\M}_n ):=\Gamma_{+,n}.
  \end{equation*}
\end{theorem}
\begin{proof}
  Let $A\in\UTz$ and $\zeta_n^k=\det(A)$.
  Since $n$ is not twice a prime power,
  $\Nm_{K_n/\Q}(\zeta_n+1)=1$; see \cite[Lemma 3.9]{IJKLZ2}.  So in
  Definition \ref{phi} we can
  take $\beta=(\zeta_n+1)^{-k}$ since
  that $\zeta_n^k = \overline{\beta}/\beta$.  Then with $A'=\beta A$ we
  have $\det A'=\beta^2\zeta_n^k = \beta\overline{\beta}\in
  \OO_{F_n}^\times$. Hence $A'$ is of the form
  \begin{equation*}
    A'=\left[\begin{array}{rr}
        r+s\sqrt{-1} & t + u\sqrt{-1}\\
        -t+u\sqrt{-1} & r-s\sqrt{-1}
      \end{array}\right]
  \end{equation*}
  and 
  \begin{equation*}
    \ovn([A])=[r-ui-tj-sk]\in\Gamma_{+}(\widetilde{\M}_{n}).
  \end{equation*}
\end{proof}

By Proposition \ref{proj}\eqref{cc}, we always have
$[\PU_2^\zeta(R_n):\PSUT(R_n)]=2$, and thus $\ovn|_{\PUTz(R_n)}$
is an isomorphism precisely when
$[\Gamma_{+,n}:\Gamma_{1,n}]=2$.  The condition for this is described in Proposition
\ref{house}\eqref{house4}.
Also notice that when $n$ is a power of $2$, we have
$\PU_2^\zeta(R_n)=\PUT(R_n)$ from Proposition \ref{hearts}.

\subsection
    [The tree for $\SLT(F_{n,\p})$
via maximal orders in $\H_n$]
    {\texorpdfstring{The tree for 
\protect{\boldmath{$\SLT(F_{n, \p})$}}
via maximal orders in \protect{\boldmath{$\H_n$}}}{The tree
via maximal orders in the Hamilton quaternions}}
\label{trees}

Throughout this subsection and the next we assume
Hypothesis \ref{assume}.

Fix $n$ and let $\Delta = \Delta_{\mathfrak{p}_n}$ be the
Bruhat-Tits tree 
for $\SL_2(F_{n, \p})$ with $F_{n, \p}$ the completion of
  $F_n$ at the prime $\p = \p_n$. Generalizing the
  discussion in 
  \cite[Sect.~4]{K}, we may describe $\Delta$ in the following manner. Let
  $\M_n$ be the maximal $\uOO$-order \eqref{eq: Mn} of $\H = \H_n$.
The vertices $\Verts(\Delta)$ are identified with 
  the maximal $\uOO$-orders $\M$
  for which $\M_v := \M\otimes_{\uOO} \uOO_{v} = 
  \M_{n, v} := \M_n \otimes_{\uOO} \uOO_{v}$ for every place
  $v \ne \p$ of $\uOO_n$.
  In this section we will identify the vertex $\bv\in\Ver(\Delta)$
  with its corresponding
  maximal order $\M_{\bv}$.
The edges originating from a vertex $\M \in \Verts(\Delta)$
correspond to left $\M$-ideals of norm $\p$. The
edge corresponding to an ideal $I$ terminates at the right order of $I$. There
are $\Norm(\p) + 1$ edges originating from each
vertex. 

Let $\be$ be an edge originating at a vertex $\M$ and terminating at $\M'$,
corresponding to the ideal $I$.
The opposite edge $\bar \be$ then corresponds to the left
$\M'$-ideal $\bar{I} = \fp I^{-1}$ where $I^{-1} = \{ \alpha
\in \H: I\alpha \subset \M\}$ and we have $\Norm(\bar{I})=\fp$.
Clearly $\M', \M$ are the left and right
orders of $\bar I$ and $\bar{\bar I} = I$.

The  undirected edge between $\M$ and
$\M'$ is identified with the Eichler order
$\mathcal{E} = \M
\cap \M'$ of level $\p$. The connection between the pair of directed
edges $I, \bar{I}$ and $\mathcal{E}$ is that $\mathcal{E} = I +
\bar{I}$ while $I$, respectively  $\bar{I}$, is the unique maximal left
$\M$-, respectively $\M'$-, ideal in
$\mathcal{E}$.  To
see this, conjugate and choose a basis locally at $\p$ so that
$$ \mathcal{E}_{\p} = \begin{bmatrix} \uOO_{\p} &
 \uOO_{\p}\\
\fp\uOO_{\p} &
\uOO_{\p}\end{bmatrix},$$
where $\uOO_{\p}$ is the ring of integers in
$F_{n_\p}$.  (Recall that $\H$ is unramified at $\p$, so
$\H \otimes_{F_n} F_{n,\p} \cong
\Mattt(F_{n,\p})$.)  Switching labels if necessary we see that  
$$ \M_{\p} = \begin{bmatrix} \uOO_{\p} &
  \uOO_{\p}\\
\uOO_{\p} & \uOO_{\p}\end{bmatrix}\,
,\quad \M_{\p}' = \begin{bmatrix} \uOO_{\p} &
  \fp^{-1}\uOO_{\p}\\
\fp\uOO_{\p} &
\uOO_{\p}\end{bmatrix},$$
are the two maximal orders containing $\mathcal{E}_{\p}$.

Simple calculations show that
$$ I_{\p} = \begin{bmatrix} \fp\uOO_{\p} &
 \uOO_{\p}\\
\fp\uOO_{\p} &
\uOO_{\p}\end{bmatrix}, \quad\bar{I}_{\p} = \begin{bmatrix} \uOO_{\p} &
 \uOO_{\p}\\
\fp\uOO_{\p} &
\fp\uOO_{\p}\end{bmatrix}.$$
Note that $\mathcal{E}_{\p} = I_{\p} + \bar{I}_{\p}$ and that $\M_{\p}$, $\M_{\p}'$
are the left and right orders of $I_{\p}$, while the reverse
holds for $\bar{I}_{\p}$. Moreover, $\Norm(I_{\p}) =
\Norm(\bar{I}_\p) = \p$ and
$$ I_{\p}^{-1} = \begin{bmatrix} \fp^{-1}\uOO_{\p} &
 \fp^{-1}\uOO_{\p}\\
\uOO_{\p} &
\uOO_{\p}\end{bmatrix},$$
so that $\bar{I}_{\p} =\fp I_{\p}^{-1}$.

Each edge in the tree is given length $1$. As usual, the distance
$\dist(\bv,\bw)$ between two
vertices $\bv$, $\bw$ is the length of the shortest path between them.

\subsection[The Clifford-cyclotomic group in
  $\overline{\Gamma}_n\subseteq \Pp\!\H^\times$]
           {\texorpdfstring{The Clifford-cyclotomic group in
  \protect{\boldmath{$\overline{\Gamma}_n\subseteq
        \Pp\!\H^\times$}}}{The Clifford-cyclotomic group and the Hamilton
quaternions}}
\label{vet}

Let the vertex $\bv\in \Delta$ be such that
$\M_{\bv}=\M$.
In what follows we assume that the ideal $(2)$
in $F_n$ is the square of a principal ideal: $(2)=(\alpha:=\alpha_n)^2$.
Let $[h]$, $[t_n]$ be as in \eqref{eq: H, Tn}.
For $\bv_1,\bv_2\in\Ver(\Delta)$, let $P(\bv_1,\bv_2)$ be the path
in $\Delta$ between $\bv_1$ and $\bv_2$ of length
$\dist(\bv_1,\bv_2)$. 
\begin{prop}
  \label{droul}
Put $[t]:=[t_n]$ and 
  let $\bw$ be the midpoint of the path $P(\bv, [t]\bv)$.
  If $\dist(\bv, [t]\bv)$ is even,
  then $\bw\in\Ver(\Delta)$.  If $\dist(\bv, [t]\bv)$
  is odd, then $\bw$ is a vertex in the barycentric subdivision
  of $\Delta$.  Then $[t]$ fixes $\bw$.
\end{prop}
\begin{proof}
If $v = [t]v$, then the statement holds trivially.  If not, let $w'$
be a vertex fixed by $[t]$.  Then $\dist(w',v) = \dist(w',[t]v)$.  Let
$P_v$ and $P_{[t]v}$ be the shortest paths from $w'$ to $v$, $[t]v$
respectively, and let $x$ be the last vertex that is in both.  Then
the paths from $x$ to $v$, $[t]v$ obtained from $P_v$, $P_{[t]v}$ by
deleting the path from $w'$ to $x$ are the shortest paths from $x$ to
$v$, $[t]v$, and the reverse of the path from $x$ to $v$ followed by
the path from $x$ to $[t]v$ is the shortest path from $v$ to $[t]v$.
Since $\dist(x,v) = \dist(x,v')$, the claim follows.
\end{proof}

\begin{prop}
  \label{deranged}
  Assume that $n$ is not a power of $2$.
  Let $\bw''\in\Ver(\Delta)$ be such that $\M_{\bw''}=\T_n$ as in
  Theorem \textup{\ref{thm: tn}} and let $\bw$ be as in Proposition
  \textup{\ref{droul}}.
  Then $\bw''=\bw$.
\end{prop}
\begin{proof}
  By Proposition \ref{droul}, $[t]\in \overline{\Gamma}_{n}$
  fixes $\bw$.
  But $\Norm_{\H_n/F_n}(t_n)$ is
  a unit by \cite[Lemma 3.9]{IJKLZ2},
  so
  $t_n\in\M_{\bw}^\times$.  Hence $e^{2\pi k/n}=t_n-1\in\M_{\bw}$.  We
  also know $j$ is in both maximal orders $\M_{\bw''}=\T_n$ and
  $\M_{\bv}=\M$.  Therefore $j$ is in each maximal order corresponding
  to vertices in the path $P(\bw'',\bv)$.  In particular, applying the
  same argument as in the proof of Proposition \ref{droul} with $\bw''$
  replacing $\bw'$, we get that $\bw$ is on $P(\bw'',\bv)$.  Therefore
  both $j$ and $e^{2\pi k/n}$ are in $\M_\bw$ and hence
  $\M_\bw=\M_{\bw''}$.
\end{proof}
\begin{prop}  
  \label{pointer}
  We have
  \begin{enumerate}[\upshape (a)]
  \item
    \label{pointer1}
    $ (t^{n/2}h)^2=u\cdot j$ for $u\in\uOOn^\times$.
    Hence
    \[\overline{\Gamma}_{n,\bw}\cong D_n\subseteq \langle [h], [t]\rangle
    \cong \Pp\!\Gg_n .
    \]
  \item
    \label{pointer2}
    Conjugation by $(1+k)/\sqrt{2} \in \Pp\!\M^\times\cong \Sc_4$ is a $4$-cycle 
    and conjugation by $(-i-k)/\sqrt{2}$ is a transposition such that the
    product has order $3$, and
    \[
    \overline{\Gamma}_{n,\bv}\cong \Sc_4 \subseteq \langle [h],[t]\rangle
    \cong\Pp\!\Gg_n.
    \]
    \end{enumerate}
  \end{prop} 
  \begin{proof}
    Statement (a) is already contained in the proof of Theorem \ref{thm: tn}.

    As for (b), the identification with $\Sc_4$ is a standard fact
    for the Hamilton quaternions over $\Q$.  In particular, 
    the conjugations by $(\pm r \pm s)$,
    where $r \ne s$ and $\{r,s\} \subset \{i,j,k\}$, are of order $2$, since
    $(\pm r \pm s)^2 = -2$ is central.
    There are $6$ of them, so these form a conjugacy
    class, which must be that of the transpositions.  
    On the other hand, the conjugations by $1 \pm q$, for 
    $q \in \{i,j,k\}$, are of order $4$, since $(1 \pm q)^4$ is a scalar but
    $(1 \pm q)^2$ is not.  It is easily checked that the product
    of one of each type is of order $2$ if and only if $\{r,s,t\} = \{i,j,k\}$.
    Since one element of order $4$ and one of order $2$ generate $\Sc_4$ if
    and only if their product has order $3$, it follows that conjugation by
    $(1+k)/\sqrt{2}$ and $(-i-k)/\sqrt{2}$ generate
    $\Sc_4$.  Since $t^{n/4} = (1+k)/\sqrt{2}$ up to scalars, conjugation by
    $t^{n/4}$ is the same as conjugation by $(1+k)/\sqrt{2}$, while conjugation 
    by $h = -i-k$ is the same as conjugation by $(-i-k)/\sqrt{2}$, so the claim
    that $\overline{\Gamma}_{n,\bv}\cong \Sc_4 \subseteq \langle [h],[t]\rangle$
    follows as well.
  \end{proof}



\section{Introduction to the examples}
\label{part: examples}

Using Theorem~\ref{theorem: reconstruct} and the description of the
Bruhat-Tits tree $\Delta = \Delta_n$ in terms of maximal orders
in $\H_n$ in \cite[Section 4]{K},
we compute $\PSUT(R_n), \PUTz(R_n)$, and
$\PUT(R_n)$ as the
fundamental group of explicit graphs of groups for $n =
8,12,16,20, 24,28,32, 36, 40,48$ arising from the actions
of $\Gamma_1, \Gamma_+$, and $\Gamma_0$ (respectively) on $\Delta$. 
Our computations, which are done using
Magma~\cite{BCP},  do not give
enough information to present $\PSUT(R_{60})$ or $\PUTz(R_{60})$ in this manner,
although they do for $\PUT(R_{60})$. Our attempted computations for
$n = 52,56$ did not finish in the small amount of time we allotted and we deemed all $n
> 60$ too costly to try. These computations give a presentation of these groups as amalgamated
products.

All computations are done using Magma's quaternion algebra machinery
in $\H_n$  by exploring the
neighborhood around the standard maximal order $\M_n$ (Section \ref{sec:Hamilton quaternions}). This order has
$\M_n^*/\left(\mathcal{O}_n^+\right)^* \simeq S_4$, 
thus rooting our graphs at a vertex common (in all the computed examples) to $\PGg_n$. This is
important for $n =28$ and $60$. These are the two cases where there
exist some types of maximal orders that are not connected to
$\M_n$. In our graphs, the vertex corresponding to $\M_n$ is
outlined in black. It has stabilizer subgroup $S_4$.

In all our examples $\T_n$ (see statement of Theorem \ref{thm: tn})  represents a vertex in the h-graph of groups for $\PUT(R_n)$ with
stabilizer subgroup isomorphic to $D_n$. There is also a path
(without backtracking) from the node represented by $\M_n$ to that
represented by $\T_n$ whose fundamental group $G$ telescopes to $G
\simeq S_4*_{D_4}D_n$.

Let $h$ be as in (\ref{rem: h}) and  $t_n$ be as in the
statement of Theorem \ref{thm: tn}. These generate $\PGg_n$. Since $h \in
\M_n$ and $t_n \in \T_n$, $\PGg_n$ embeds into $G$. This latter group is isomorphic to $\PGg_n(R_n)$ by
Theorem \ref{bird}. Theorem~\cite[Theorem 1.2(1)]{IJKLZ} tells us that these two groups
must be equal for $n =  8,12, 16, 24$. We do not know if this holds
for the remaining examples $n = 20, 28, 32, 36,40, 48, 60$. In any,
case, using Proposition \ref{prop:inf-index} we have the new result
$\PGg_n < G\ll [\PUTz(R_n)]_{f}$ for $n = 20,28,32$.

We can also compute the Euler-Poincar\'{e} characteristics from the
associated graph of groups. 
\begin{definition1}\label{defn: mass} Let $\Gr = (\Gamma, \gr)$ be an
  h-graph of groups with finite vertex and edge isotropy
  groups. Define the mass of a vertex $\bv \in \gr$ to be $\m(\bv)
  =1/\#\Gamma_{\bv}$. The mass of an edge $\be \in \Ed(\gr)$ is
  $m(\be) =1/(2\#\Gamma_{\be})$. 
  The {\it
    vertex mass} of $\Gr$ is
  $$\VM(\Gr) := \VM(\Gamma, \gr) := \sum_{\bv \in
    \Ver(\gr)}\m(\bv)\,.$$ Its {\it edge mass} is $$\EM(\Gr) :=
  \EM(\Gamma, \gr) := \sum_{\be \in \Ed(\gr)}\m(\be)\, .$$
\end{definition1}
\begin{remark1}\label{rem: edge mass}
  Recall that, in our graphs, edges $\be$ have opposite edges
  $\widebar{\be}$ so the mass attached to the geometric edge 
  $\{\be,\widebar{\be}\}$ with $\be \ne \widebar{\be}$ is $1/\#\Gamma_{\be} =
  1/\#\Gamma_{\widebar{\be}}$. Our h-graphs of groups are
  actually graphs of groups (being a quotient of the tree after
  subdividing inverted edges). In the pictures we draw we elide the
  inserted vertices to make clear the associated h-graph. These
  elided vertices still contribute to the vertex mass of the graph in the
  usual way. 
  \end{remark1}
By \cite[Theorem 2.20]{IJKLZ2}, the Euler-Poincar\'{e} characteristic of the
fundamental group $\Gr$ of a graph of groups is given by
\begin{equation}
  \label{chi}
\chi(\pi_1(\Gr)) = \VM(\Gr) - \EM(\Gr) .
\end{equation}
On the other hand, if we have a group presented as an amalgamated
product, it is easy to compute its Euler-Poincar\'{e} characteristic:
If $A$, $B$ are finite groups with $C\leq A$ and $C\leq B$, then
\begin{equation}
  \label{form}
  \chi(A\ast_{C}B)=\frac{1}{\# A}+\frac{1}{\#B}-\frac{1}{\#C}
\end{equation}
by \cite[Corollaire 1, p.~104]{S3}.

\subsection{Tables}
  Before treating examples in detail, we arrange in tabular form the
  results of some Magma computations. 
  We first summarize the notation used in the tables, using the
  notation of Section \ref{sec:cyc-rings-fields}.  All of our examples
  will satisfy the equivalent conditions of Proposition
  \ref{one/unramified} with one prime $\fp$ of $F=F_n$ lying
  above~$2$.  In addition, all but $n=28, 60$ satisfy the equivalent
  conditions Proposition \ref{punt1} having a unique prime above $2$
  in $K=K_n$. Let $q_n = 2^{f(\fp)}$ be the order of the residue field
  of $\fp$ and let $\Delta_n$ be the Bruhat-Tits tree associated to
  $\SLT(F_\fp)$. It is a regular tree of degree $q_n
  +1$.
  We compute with $\Delta_n$ using
  the maximal orders in $\H_n$ as described in \S\ref{trees}.
  
    The table below gives the results on $F_n$ we will need
    in analyzing our examples.
    \begin{table}[H]
      \label{fielddata}
      \caption{Data on the fields $F_n=\Q(\zeta_n)^+$}
    \begin{center}
    \begin{tabular}{r|c|c|c|c|c|c}
      $n$ & $\varphi(n)$ & $q_n +1$ & $\hnar(\OO_n)$
      & $\hnar(\OO^\fp_n)$ & $\#\frac{\Cln(\OO_n^\fp)[2]}
      {\Prin(\OO_n^\fp)}$ & $[\fp]\in \Cln(\OO_n)^2$?\\ \hline
      $8$ & $4$ &   $3$ &  $1$ & $1$ & $1$ & yes\\ \hline
      $12$ & $4$ & $3$  &  $2$ & $1$ & $1$ & no\\ \hline
      $16$ & $8$ &  $3$ &  $1$ & $1$ & $1$ & yes\\ \hline
      $20$ & $8$ & $5$ & $2$ & $1$ & $1$ & no \\ \hline
      $24$ & $8$ &  $3$ &  $2$ & $1$ & $1$ & no \\ \hline
      $28$ & $12$ & $9$ & $2$ & $2$ & $1$ & yes \\ \hline
      $32$ & $16$ & $3$  & $1$ & $1$ & $1$ & yes\\ \hline
      $36$ & $12$ &  $9$ & $2$ & $1$ & $1$ & no\\ \hline
      $40$ & $16$ & $5$ & $2$ & $1$ & $1$ & no \\ \hline
      $48$ & $16$ &  $3$ & $2$ & $1$ & $1$ & no \\ \hline
      $60$ & $16$ &  $17$ & $2$ & $2$ & $1$ & yes
    \end{tabular}
    \end{center}
    \end{table}

 The table below gives the information about $\H_n$ we will need
 in analyzing our examples.
 \begin{table}[H]
   \caption{Data on the quaternion algebras $\H_n=\H\otimes_{\Q}F_n$}
\begin{center}
\begin{tabular}{r|c|c|c}
      & $t(\H_n)  = h_{[\fp]\text{-rel}}(\H_n) $ && \\
      $n$ & $\quad = \hrel(\H_n) = h(\H_n)$ 
& $\#\frac{\Gamma_{+}(\M_n)}{\Gamma_{1}(\M_n)}$ &
      $\#\frac{\Gamma_{0}(\M_n)}{\Gamma_+(\M_n)}$ 
      \\ \hline
      $8$ & $1$ & $1$ & $2$ \\ \hline
      $12$ & $2$ & $2$ & $1$ \\ \hline
      $16$ & $2$ & $1$ & $2$ \\ \hline
      $20$ & $3$ & $2$ & $1$ \\ \hline
      $24$ & $3$ &  $2$ & $1$ \\ \hline
      $28$ & $5$ & $2$ & $2$ \\ \hline
      $32$ & $58$ & $1$ & $2$ \\ \hline
      $36$ & $6$ & $2$ & $1$ \\ \hline
      $40$ & $25$ & $2$ & $1$ \\ \hline
      $48$ & $39$ & $2$ & $1$ \\ \hline
      $60$ & $9$  & $2$ & $2$ 
\end{tabular}
\end{center}
 \end{table}

We now explain the equalities in the column headings of the table.
\begin{prop}\label{explain-table1}
For all $n$ in the range of the table we have
$t(\H_n) = h(\H_n) = \hrel(\H_n) = h_{[\fp]\textup{-rel}}(\H_n)$.
\end{prop}

\begin{proof}  Within the range of the table,
the class group of $R_n$ is always trivial.  So $\hrel(\H_n)$,
the number of orbits of the class group on $\Cl(\H_n)$, 
is equal to $\#\Cl(\H_n) = h(\H_n)$, and likewise for
$h_{[\fp]\text{-rel}}(\H_n)$, the number of orbits of the subgroup generated by
$\fp$.

By \cite[Exercise 17.3, Lemma 17.4.8]{Vo}, the set of right orders of the $I_i$ contains a representative of every
isomorphism class of maximal orders, where $I_1, \dots, I_{h(\H_n)}$ is a set of
representatives for the left ideal classes of a fixed maximal order.
Thus $t(\H_n) \le h(\H_n)$.
In the range of the table it can be computed that $t(\H_n) = h(\H_n)$.
\end{proof}

The following table shows the identifications of $\PSUT(R_n)$,
$\PUTz(R_n)$, and $\PUT(R_n)$ with $\Gamma_{1, n}, \Gamma_{+,n}$,  and
$\Gamma_{0, n}$.
 \begin{table}[H]
   \caption{Arithmetic discrete subgroups}
\begin{center}
\begin{tabular}{r|c|c|c}
      $n$ & $\PSUT(R_n)$ 
& $\PUTz(R_n)$ &
      $\PUT(R_n)$ 
      \\ \hline
      $8$ & $\Gamma_1= \Gamma_+$ & $\Gamma_0$ & $\Gamma_0$ \\ \hline
      $12$ & $\Gamma_1$ & $\Gamma_+ =\Gamma_0$ & $\Gamma_+ =\Gamma_0$ \\ \hline
      $16$ & $\Gamma_1=\Gamma_+$ & $\Gamma_0$ & $\Gamma_0$ \\ \hline
      $20$ & $\Gamma_1$ & $\Gamma_+ = \Gamma_0$ & $\Gamma_+ =\Gamma_0$ \\ \hline
      $24$ & $\Gamma_1$ &  $\Gamma_+ = \Gamma_0$ & $\Gamma_+ = \Gamma_0$ \\ \hline
      $28$ & $\Gamma_1$ & $\Gamma_+$ & $\Gamma_0$ \\ \hline
      $32$ & $\Gamma_1= \Gamma_+$ & $\Gamma_0$ & $\Gamma_0$ \\ \hline
      $36$ & $\Gamma_1$ & $\Gamma_+ = \Gamma_0$ & $\Gamma_+ =\Gamma_0$ \\ \hline
      $40$ & $\Gamma_1$ & $\Gamma_+ = \Gamma_0$ & $\Gamma_+ = \Gamma_0$ \\ \hline
      $48$ & $\Gamma_1$ & $\Gamma_+ =\Gamma_0$ & $\Gamma_+ = \Gamma_0$ \\ \hline
      $60$ & $\Gamma_1$  & $\Gamma_+$ & $\Gamma_0$ 
\end{tabular}
\end{center}
 \end{table}

\begin{remark1}
{\rm 
The quotient graph of groups
  $\SUT(R_n)\dbs(\SUT(R_n), \Delta_n)$, and hence its fundamental
  group $\SUT(R_n)$,  can be constructed from
  \[
  \PSUT(R_n)\dbs (\PSUT(R_n), \Delta_n)
  \]
  by inflating each edge and
  vertex group by a central $\pm 1$. Thus, amalgam presentations
  of $\SUT(R_n)$ can be constructed from those we give for
  $\PSUT(R_n)$.
  }
\end{remark1}
\subsection{Notation and Basic Results for Groups}
\label{rounded}
For a group $\Gamma$ acting on a tree $\Delta$ we write $\Gr(\Gamma):=
\Gamma\dbs(\Gamma, \Delta_{\Gamma})$, where, as in Section
\ref{groupgraph}, $\Delta_{\Gamma}$ is the graph obtained from
$\Delta$ by subdividing exactly those edges which are inverted by $\Gamma$.

\begin{definition}  We use the following conventions to label the vertex and 
edge groups of a graph of groups. 
For an integer $n$, we use $C_n$, or simply $n$, to denote the cyclic
group and $D_n$ for the dihedral group of order $2n$. For edge groups
$C_1$ we elide the label altogether.
For an even integer $n$,
the {\em quaternion group} $Q_{2n}$ of order $2n$ is the
subgroup of $\H(\BR)$ generated by $e^{2\pi i/n}$ and $j$.
It is easy to show that $Q_{2n}$ is the unique extension of $\Z/n\Z$ by
$\Z/2\Z$ acting by $x \to x^{-1}$ that is not a semidirect product, and
that $Q_{2n}/\{\pm 1\} \isisom D_{n/2}$.  In addition, we denote the binary
tetrahedral, octahedral, and icosahedral groups 
\cite[Th\' eor\` eme I.3.7]{V} by $E_{24}, E_{48}$, and $E_{120}$ respectively.
Note that $|Q_n| = n$ and $|E_n| = n$ but $|D_n| = 2n$.  Recall that
$\Gg$ or $\Gg_n$ denotes the subgroup of $\UT$ generated by $H$ and $T_n'$.

Below we give the key to our notation and conventions in the pictures
of the quotient graphs in Sections \ref{until} and \ref{until2}.
For $n<32$ we draw the graphs of the quotients of $\Delta$
by $\PUT(R_n)$, $\PUTz(R_n)$, and $\PSUT(R_n)$.
For $n\geq 32$ we only draw the graph of $\overline{\gr}_n
=\Gamma_{n,0}\backslash\Delta=
\PUT(R_n)\backslash\Delta$ due to space constraints.

\begin{key}
  \label{key}
  {\rm
Vertices and edges of quotient graphs are labeled by their
corresponding stabilizer groups.  A (nonelided) vertex of the quotient graph of
$\PUT(R_n)$ or $\PUTz(R_n)$
is indicated by a square if it is ramified in the cover from $\PSUT(R_n)$
and by a circle if it is unramified there.
The graphs $\Gamma_+\backslash \Delta$ and $\Gamma_1\backslash \Delta$
are bipartite; the bipartition on vertices is shown using red and blue.
The vertex of $\PUT(R_n)\backslash\Delta$ lying below $\bv\in\Ver(\Delta)$
with $\M_\bv =\M$ as in Section \ref{vet} is marked with an M.
Likewise the vertex or elided half-edge vertex
of $\PUT(R_n)\backslash\Delta$ lying below $\bw\in \Ver(\Delta)$
as in Proposition \ref{droul} is marked with a T.
The sub graph-of-groups $\overline{P}_n$ in $\PUT(R_n)\backslash \Delta$
with $\pi_{1}(\overline{P}_n)\cong \Pp\!\Gg_n\cong S_4\ast_{D_4}D_n$
as in Remark \ref{inj pi}
is shown in the picture of $\PUT(R_n)\backslash\Delta$ with magenta edges.
}
\end{key}
\end{definition}

\section[The $n =
        8,12,16,24$ examples:
$\Pp\!\Gg_n = \PUTz(R_n)=\PUT(R_n)$]{
\texorpdfstring{$\text{The } n=8,12,16,24 \text{ examples}:
\Pp\!\Gg_n = \PUTz(R_n)=\PUT(R_n)$}{The n= 8, 12, 16, 24 examples}}
\label{until}
For $n = 8,12,16,24$ it is known that $\Pp\!\Gg_n = \PUT(R_n)$
(see \cite[Theorems 1.1, 1.2]{IJKLZ} for references)
and that $\Pp\!\Gg_n\cong S_4\ast_{D_4}D_n$ by \cite{RS}
(see \cite[Theorem 5.1]{IJKLZ}).
We establish via quotient graphs in this section that
$\Pp\!\Gg_n=\PUT(R_n)$ and $\Pp\!\Gg_n\cong S_4\ast_{D_4}D_n$
for $n=8, 12, 16, 24$.  We also compute the Euler-Poincar\'{e}
characteristics $\chi(\PUT(R_n))$, $\chi(\PSUT(R_n))$, and
$\chi(\PUTz(R_n))$ from our $\Grn$ and $\oGrn$
for $n=8,12,16, 24$.
In all cases the answers agree with \cite[Theorem 6.6]{IJKLZ},
giving a good check on our quotient graphs.

\subsection[$n=8$]
    {\texorpdfstring{{\boldmath{$n=8$}}}{n=8}}
    \label{eat}
   We have $\PSUT(R_8) = \Gamma_{8, 1} = \Gamma_{8,+}$ and
   $\PUTz(R_8)= \PUT(R_8) = \Gamma_{8,0}$.
 The quotient graph of groups $\Gr_{\! 8}$ 
for $\PSUT(R_8)$ is
\begin{center}
\begin{tikzpicture}
\draw[ultra thick,color = green] (3.4,0)--(5.6,0);
\draw[ultra thick,color = blue] (3, 0) circle [radius = 0.4];
\draw[ultra thick,color = red] (6, 0) circle [radius = 0.4];
\node at (3,0){$S_4$};
\node at (1.5,0){$\Gr_{\! 8}$:};
\node at (3.4,0.4){{\tiny $M$}};
\node at (6,0){$S_4$};
\node at (4.5,0.2){$D_4$};
\node at (6.6,-0.2){.};
\end{tikzpicture}
\end{center}
From $\Gr_{\! 8}$  we compute the Euler-Poincar\'{e} characteristic
$$\chi\left(\PSUT(R_8)\right) = 1/24 + 1/24 - 1/8 = -1/24$$
and the amalgam $\PSUT(R_8) = \pi_1(\Gr_{\! 8})=S_4*_{D_4}S_4$.

The quotient h-graph of
groups $\oGr_{\! 8}$ for $\PU_2^\zeta(R_8)=\PUT(R_8)$ is
\begin{center}
\begin{tikzpicture}
\draw[ultra thick,color = fuchsia] (0.4,0)--(1.5,0);
\draw[ultra thick,color = blue] (0, 0) circle [radius = 0.4];
\node at (0,0){$S_4$};
\node at (-1.5,0){$\oGr_{\! 8}$:};
\node at (0.4,0.4){{\tiny $M$}};
\node at (0.95,0.2){$D_4$};
\node at (2,0){$D_8$};
\node at (2.3,0.3){{\tiny $T$}};
\node at (2.6,-0.2){.};
\end{tikzpicture}
\end{center}
From $\oGr_{\!8}$ we compute using \eqref{chi}
$$\chi(\PUTz(R_8)) =\chi(\PUT(R_8)) = 1/24 + 1/16
-1/8 = -1/48 $$
as well as the amalgam $\PUT(R_8)=\pi_1(\oGr_{\! 8})=S_4\ast_{D_4}D_8$.
We see that $\PG_8=\PUT(R_8)$ and we hence recover
$\PG_8\cong S_{4}\ast_{D_4}D_8$.
\subsection[$n=12$]
    {\texorpdfstring{\protect{\boldmath{$n=12$}}}{n=12}}\label{sec:nequals12}
     We have $\PSUT(R_{12}) = \Gamma_{12, 1}$ and
   $\PUTz(R_{12})= \PUT(R_{12}) = \Gamma_{12, +}=\Gamma_{12,0}$.
     In this case the double cover $$\Gamma_{12,1}\backslash \Delta
     \longrightarrow \Gamma_{12,+}\backslash \Delta =\Gamma_{12,0}
     \backslash \Delta$$
is not \'{e}tale.
The quotient graph of groups $\Gr_{12}$ for $\PSUT(R_{12})$ is
\begin{center}
\begin{tikzpicture}
\draw[ultra thick,color = green] (3.4,0)--(5.6,0);
\draw[ultra thick,color = blue] (3, 0) circle [radius = 0.4];
\draw[ultra thick,color = red] (6, 0) circle [radius = 0.4];
\node at (1.5,0){$\Gr_{12}$:};
\node at (3,0){$A_4$};
\node at (3.4,0.4){{\tiny $M$}};
\node at (6,0){$D_6$};
\node at (4.5,0.2){$D_2$};
\node at (6.6,-0.2){,};
\end{tikzpicture}
\end{center}
giving
$\PSUT(R_{12}) \isisom \pi_1(\Gr_{12})\cong A_4 *_{D_2} D_6$ and
$\chi(\PSUT(R_{12})) = -1/12$.

The quotient h-graph of groups for $\PU_2^\zeta(R_{12})=\PUT(R_{12})$ is   
\begin{center}
\begin{tikzpicture}
  \draw[ultra thick,color = fuchsia] (3.35,0)--(5.55,0);
\node at (1.5,0){$\oGr_{12}$:};
  \node[draw,ultra thick,color = blue] at (3,0){\color{black} $S_{4}$};
\node[draw,ultra thick,color = red] at (6,0){\color{black} $D_{12}$};
\node at (3.5,0.4){{\tiny $M$}};
\node at (6.6,0.4){{\tiny $T$}};
\node at (4.5,0.2){$D_{4}$};
\node at (6.8,-0.2){,};
\end{tikzpicture}
\end{center}
from which we derive $\PUTz(R_{12}) = \PUT(R_{12})\isisom
\pi_1(\oGr_{12})
\isisom S_4 *_{D_4} D_{12}$,
$$\chi(\PUTz(R_{12})) = \chi(\PUT(R_{12})) = -1/24,$$
and $\PG_n=\PUT(R_{12})\cong S_4\ast_{D_4}D_{12}$.

\subsection[$n=16$]
{\texorpdfstring{\protect{\boldmath{$n=16$}}}{n=16}}
     Here $\PSUT(R_{16}) = \Gamma_{16, 1}=\Gamma_{16, +}$ and
   $\PUTz(R_{16})= \PUT(R_{16}) = \Gamma_{16, 0}$.

     The  graph of groups $\Gr_{16}$
     and the corresponding amalgam for $\PSUT(R_{16})\isisom\pi_1(\Gr_{16})$ are
\begin{center}
\begin{tikzpicture}
\draw[ultra thick,color = green] (0.4,0)--(2.6,0);
\draw[ultra thick,color = green] (3.4,0)--(5.6,0);
\draw[ultra thick,color = green] (6.4,0)--(8.6,0);
\draw[ultra thick,color = blue] (0, 0) circle [radius = 0.4];
\draw[ultra thick,color = red] (3, 0) circle [radius = 0.4];
\draw[ultra thick,color = blue] (6, 0) circle [radius = 0.4];
\draw[ultra thick,color = red] (9, 0) circle [radius = 0.4];
\node at (-1.5,0){$\Gr_{16}$:};
\node at (0,0){$S_4$};
\node at (0.4,0.4){{\tiny $M$}};
\node at (3,0){$D_8$};
\node at (6,0){$D_8$};
\node at (9,0){$S_4$};
\node at (1.5,0.2){$D_4$};
\node at (4.5,0.2){$D_8$};
\node at (7.5,0.2){$D_4$};
\end{tikzpicture}
\end{center}
$$\PSUT(R_{16}) \isisom S_4 *_{D_4} D_8 *_{D_4}
S_4 .$$
We compute
$\chi\left(\PSUT(R_{16})\right) = -5/48$ using \eqref{chi}.

The quotient h-graph of groups $\oGr_{16}$
for  $\PUTz(R_{16})=\PUT(R_{16})$ is  
\begin{center}
\begin{tikzpicture}
\draw[ultra thick,color = fuchsia] (0.4,0)--(2.6,0);
\draw[ultra thick,color = fuchsia] (3.4,0)--(4.5,0);
\draw[ultra thick,color = blue] (0, 0) circle [radius = 0.4];
\draw[ultra thick,color = blue] (3, 0) circle [radius = 0.4];
\node at (-1.5,0){$\oGr_{16}$:};
\node at (0,0){$S_4$};
\node at (0.4,0.4){{\tiny $M$}};
\node at (3,0){$D_8$};
\node at (1.5,0.2){$D_4$};
\node at (4,0.2){$D_{8}$}; 
\node at (5.2,0){$D_{16}$}; 
\node at (5.5,0.3){{\tiny $T$}};
\node at (5.8,-0.2){,};
\end{tikzpicture}
\end{center}
giving $\PUTz(R_{16}) = \PUT(R_{16}) \isisom
\pi_1(\oGr_{16})\isisom S_4 *_{D_4} D_{16}$ and
$$\chi(\PUTz(R_{16})) = \chi(\PUT(R_{16})) =
-5/96 .$$
As before we have  $\PUT(R_{16}) = \Pp \!\Gg_{16}$ from
Proposition \ref{pointer}.
\subsection[$n=24$]{\texorpdfstring{\protect{\boldmath{$n=24$}}}{n=24}}
    \label{n=24}
In this case $\PSUT(R_{24}) = \Gamma_{24,1}$ while $\PUTz(R_{24}) =
\PUT(R_{24}) = \Gamma_{24, +} = \Gamma_{24, 0}$. Again the double
cover $\PSUT(R_n)\backslash \Delta\longrightarrow
\PUT(R_n)\backslash \Delta$ is not \'{e}tale.

The  quotient graph of groups $\Gr_{24}$ for $\PSUT(R_{24})$ is
\begin{center}
\begin{tikzpicture}
\draw[ultra thick,color = green] (0.4,-1)--(2.6,0);
\draw[ultra thick,color = green] (0.4,1)--(2.6,0);
\draw[ultra thick,color = green] (3.4,0)--(5.6,0);
\draw[ultra thick,color = blue] (0, -1) circle [radius = 0.4];
\draw[ultra thick,color = blue] (0, 1) circle [radius = 0.4];
\draw[ultra thick,color = red] (3, 0) circle [radius = 0.4];
\draw[ultra thick,color = blue] (6, 0) circle [radius = 0.4];
\node at (0,-1){$S_4$};
\node at (0,1){$S_4$};
\node at (0.4,1.4){{\tiny $M$}};
\node at (-2.5,0){$\Gr_{24}$:};
\node at (3,0){$D_4$};
\node at (6,0){$D_{12}$};
\node at (1.5,-0.9){$D_4$};
\node at (1.5,0.9){$D_4$};
\node at (4.5,0.2){$D_4$};
\node at (6.6,-0.2){.};
\end{tikzpicture}
\end{center}
From this we compute
$\chi\left(\PSUT(R_{24})\right) = -1/8$ 
and $\PSUT(R_{24})$ is the amalgam
$$\PSUT(R_{24})  \isisom \pi_1(\Gr_{24})\isisom *_{D_4} \{S_4, S_4, D_{12}\}$$
of $D_{12}$ and the two copies of $S_4$ over their common subgroup
$D_4$.

Since $\Gamma_{24,0}$ acts without inversions we get the following
quotient
 graph of groups for $\PU_2^\zeta(R_{24}) = \PUT(R_{24})  = 
\Gamma_{24,+} = \Gamma_{24,0}:$ 
\begin{center}
\begin{tikzpicture}
\draw[ultra thick,color = fuchsia] (0.4,0)--(2.6,0);
\draw[ultra thick,color = fuchsia] (3.4,0)--(5.55,0);
\draw[ultra thick,color = blue] (0, 0) circle [radius = 0.4];
\node at (-1.5,0){$\oGr_{24}$:};
\node at (0,0){$S_4$};
\node at (0.4,0.4){{\tiny $M$}};
\node[draw,ultra thick,color = red] at (3,0){\color{black} $D_8$};
\node[draw,ultra thick,color = blue] at (6,0){\color{black} $D_{24}$};
\node at (6.6,0.4){{\tiny $T$}};
\node at (1.5,0.2){$D_4$};
\node at (4.5,0.2){$D_{8}$};
\node at (6.9,-0.2){.};
\end{tikzpicture}
\end{center}
Hence $\PUTz(R_{24}) = \PUT(R_{24}) \isisom \pi_1(\oGr_{24})\isisom  S_4 *_{D_4}
D_{24}\isisom \PGg_{24}$ and $$\chi(\PUTz(R_{24})) =\chi(\PUT(R_{24})) = -1/16.$$

\section[The $n =20, 28, 32, 36, 40, 48, 60$ examples:
$\PGg_n \ne \PUTz(R_n)$]
{\texorpdfstring{The $n =20, 28, 32, 36, 40, 48, 60$ examples:
$\PGg_n \ne \PUTz(R_n)$}{The n=20, 28, 32, 40, 48, 60 examples}}
\label{until2}
Now $\Gg_n \ne \UTz(R_n)$, $\PGg_n \ne
\PUTz(R_n)$, $\SUT(R_n)
\ne \SGg_n$, and $\PSGg_n \ne \PSUT(R_n)$ when $n\notin\{
8,12,16,24\}$ (see \cite[Theorems 1.1, 1.2]{IJKLZ} for references). In fact, in our examples $\PGg_n$ can be seen as
a proper subtree of the quotient h-graph of groups for
$\PUTz(R_n)$. Theorem~\ref{prop:inf-index} shows that
$\PGg_n \ll [\PUTz(R_n)]_{f}$ for these $n$.
\subsection[$n=20$]
           {\texorpdfstring{\protect{\boldmath{$n=20$}}}{n=20}}
We have the identifications $\PSUT(R_{20}) = \Gamma_{20, 1}$ and
$\PUTz(R_{20}) = \PUT(R_{20}) = \Gamma_{20, +} = \Gamma_{20, 0}$. 
The double cover $\gr_n=\Gamma_{20,1}\backslash\Delta\longrightarrow
\overline{\gr}_n=\Gamma_{20,0}\backslash\Delta$ is not \'{e}tale. 
The  quotient graph of groups $\Gr_{20}$ for $\PSUT(R_{20})$ is
\begin{center}
\begin{tikzpicture}
\draw[ultra thick,color = green] (0.4,-1)--(2.6,0);
\draw[ultra thick,color = green] (0.4,1)--(2.6,0);
\draw[ultra thick,color = green] (3.4,0)--(5.6,0);
\draw[ultra thick,color = blue] (0, -1) circle [radius = 0.4];
\draw[ultra thick,color = blue] (0, 1) circle [radius = 0.4];
\draw[ultra thick,color = red] (3, 0) circle [radius = 0.4];
\draw[ultra thick,color = blue] (6, 0) circle [radius = 0.4];
\node at (0,-1){$A_5$};
\node at (0,1){$A_5$};
\node at (-2,0){$\Gr_{20}$:};
\node at (3,0){$A_4$};
\node at (3.4,0.4){{\tiny $M$}};
\node at (6,0){$D_{10}$};
\node at (1.5,-0.9){$A_4$};
\node at (1.5,0.9){$A_4$};
\node at (4.5,0.2){$D_2$};
\node at (6.6,-0.2){.};
\end{tikzpicture}
\end{center}
So we obtain 
$\PSUT(R_{20}) \isisom \pi_1(\Gr_{20})\isisom A_5 *_{A_4} A_5 *_{D_2}
D_{10}$ and $\chi\left(\PSUT(R_{20})\right) = -1/4.$

Since $\Gamma_{20, 0}$ acts without inversions we get the
 quotient graph of groups $\oGr_{20}$ for $\PU_2^\zeta(R_{20}) =\PUT(R_{20}) =\Gamma_{20, +} = \Gamma_{20,0}$ below:
\begin{center}
\begin{tikzpicture}
\draw[ultra thick,color = green] (0.4,0)--(2.65,0);
\draw[ultra thick,color = fuchsia] (3.35,0)--(5.55,0);
\draw[ultra thick,color = blue] (0, 0) circle [radius = 0.4];
\node at (-1.5,0){$\oGr_{20}$:};
\node at (0,0){$A_5$};
\node[draw,ultra thick,color = red]  at (3,0){\color{black} $S_4$};
\node[draw,ultra thick,color = blue]  at (6,0){\color{black} $D_{20}$};
\node at (3.5,0.4){{\tiny $M$}};
\node at (6.6,0.4){{\tiny $T$}};
\node at (1.5,0.2){$A_4$};
\node at (4.5,0.2){$D_4$};
\node at (6.9,-0.2){.};
\end{tikzpicture}
\end{center}
Thus $\PUTz(R_{20}) = \PUT(R_{20})\isisom \pi_1(\oGr_{20})\isisom A_5 *_{A_4} S_4 *_{D_4}D_{20}$ and
\[
\chi(\PUTz(R_{20})) = \chi(\PUT(R_{20})) = -1/8.
\]
Proposition \ref{prop:inf-index} shows that  $\PGg_{20} \sinf [\PUT(R_{20})]_f$,
so Corollary
\ref{cor:other-inf-index} applies.  On the other hand, it is clear that
$[\PUT(R_{20})]_f = \PUT(R_{20})$.

\subsection[$n=28$]
    {\texorpdfstring{\protect{\boldmath{$n=28$}}}{n=28}}
    \label{event}
This is one of the two examples (the other is $n=60$) for which the
unique prime above $2$ in $F_n$ splits in $K_n$.
Hence $\PSUT(R_{28}) = \Gamma_{28,1}$, $\PUTz(R_{28}) = \Gamma_{28,
  +},$ and $\PUT(R_{28}) = \Gamma_{28,0}$ are all distinct with $\PUT(R_{28})/\PUTz(R_{28}) \isisom \Z/2\Z$ and
$\PUTz(R_{28})/\PSUT(R_{28}) = \Z/2\Z$.

 The class number of $\H_{28}$ is $5$, but only three types of orders are
 connected to $\M$. 
 The quotient h-graph of groups $\oGr_{28}$ for $\PUT(R_{28})
 = \Gamma_{28, 0}$ is given below:

\begin{figure}[H]
  \includegraphics[scale = 0.8,trim={0 5.13696in 0 3in}, clip]{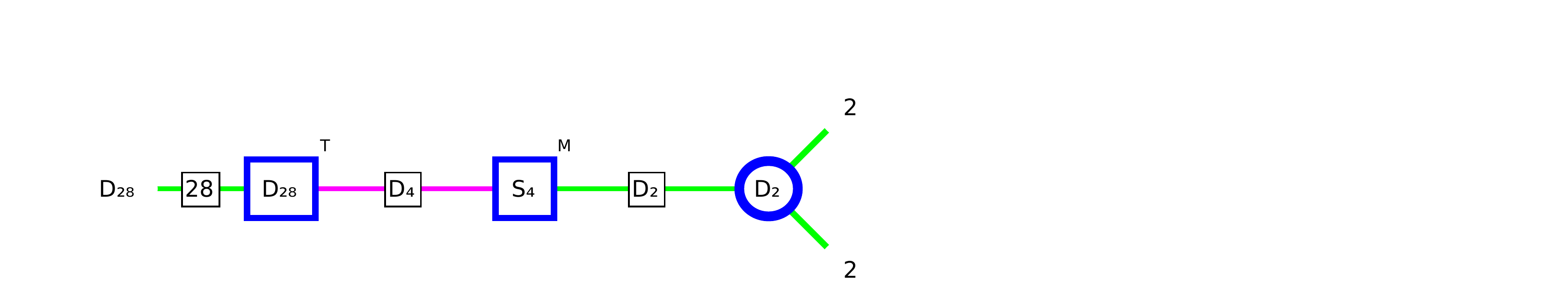}
\end{figure}
Thus
\begin{equation*}
  \PUT(R_{28}) \isisom \Gamma_{28,0}\isisom \pi_1(\oGr_{28}) = D_{28} *_{C_{28}} D_{28} *_{D_4}
  S_4 * C_2^{*2}
\end{equation*}
and $\chi(\PUT(R_{28}))=-13/12$.

 The quotient graph of groups for $\PUTz(R_{28}) = \Gamma_{28, +}$, which is the bipartite double cover of $\oGr_{28}$, is:

\begin{figure}[H]
  \includegraphics[scale = 0.8, trim={-1.3in, 3.3in, 6in, 3.1in}, clip]{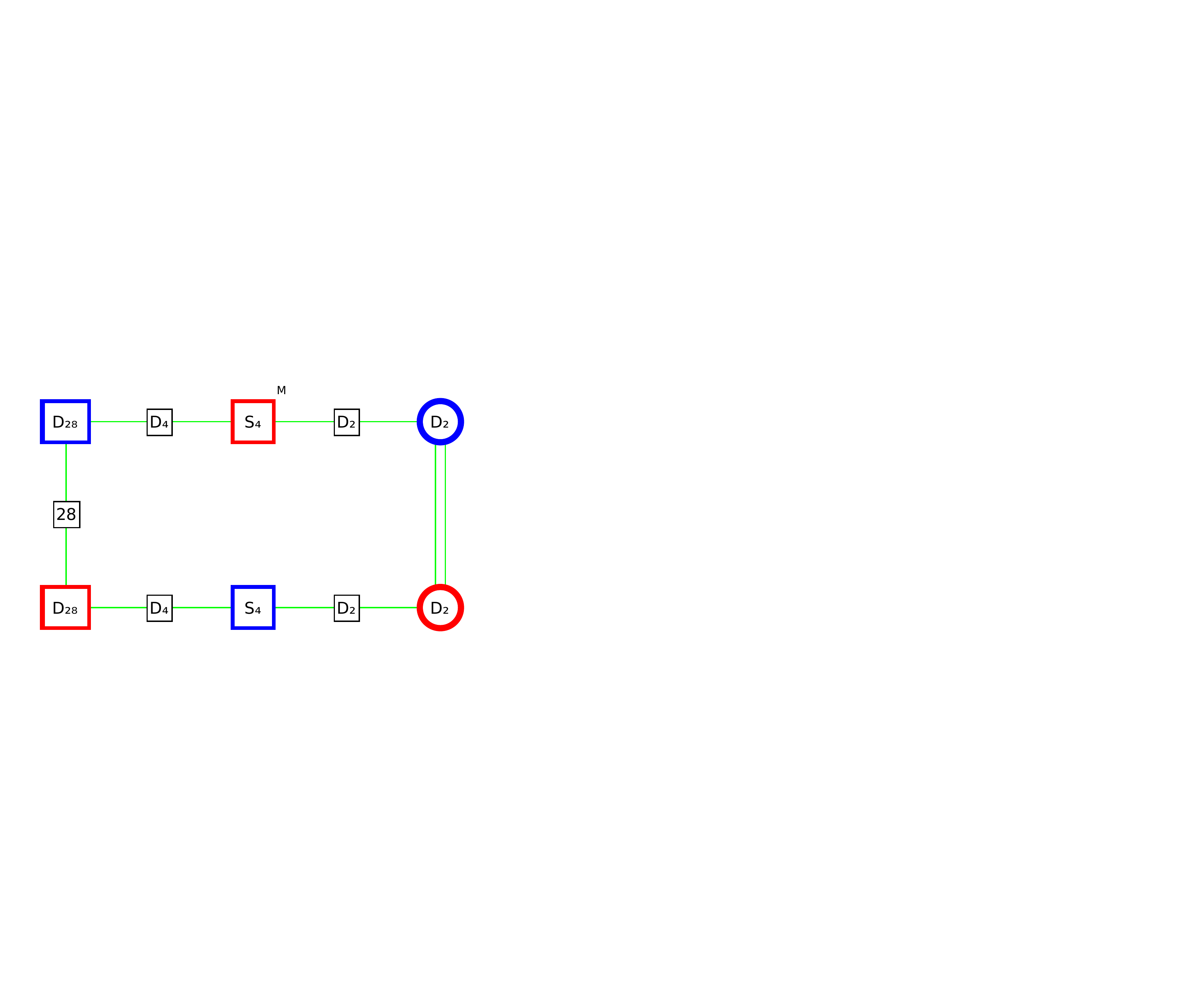}
\end{figure}

\noindent Notice that, unlike all the graphs we have dealt with above, this graph
is not  a tree; in fact it has genus $2$.  

By Theorem \ref{theorem: free generators},
\begin{equation*}\PU_2^\zeta(R_{28}) \isisom \Gamma_{28,+} =
  S_4 *_{D_4} D_{28} *_{C_{28}}
  D_{28} *_{D_4} S_4 * \Z^{*2}
\end{equation*}
and $\chi(\PU_2^\zeta(R_{28}))=-13/6$.

 The quotient graph of groups  $\Gr_{28}$ for $\Gamma_{28,1}$ is the double cover of the quotient graph
of groups for $\Gamma_{28,+}$ ramified at the vertices of that graph  marked with a
square:

\begin{figure}[H]

\includegraphics[scale = 0.8, trim={-1.15in 3in 5in 3.in}, clip]{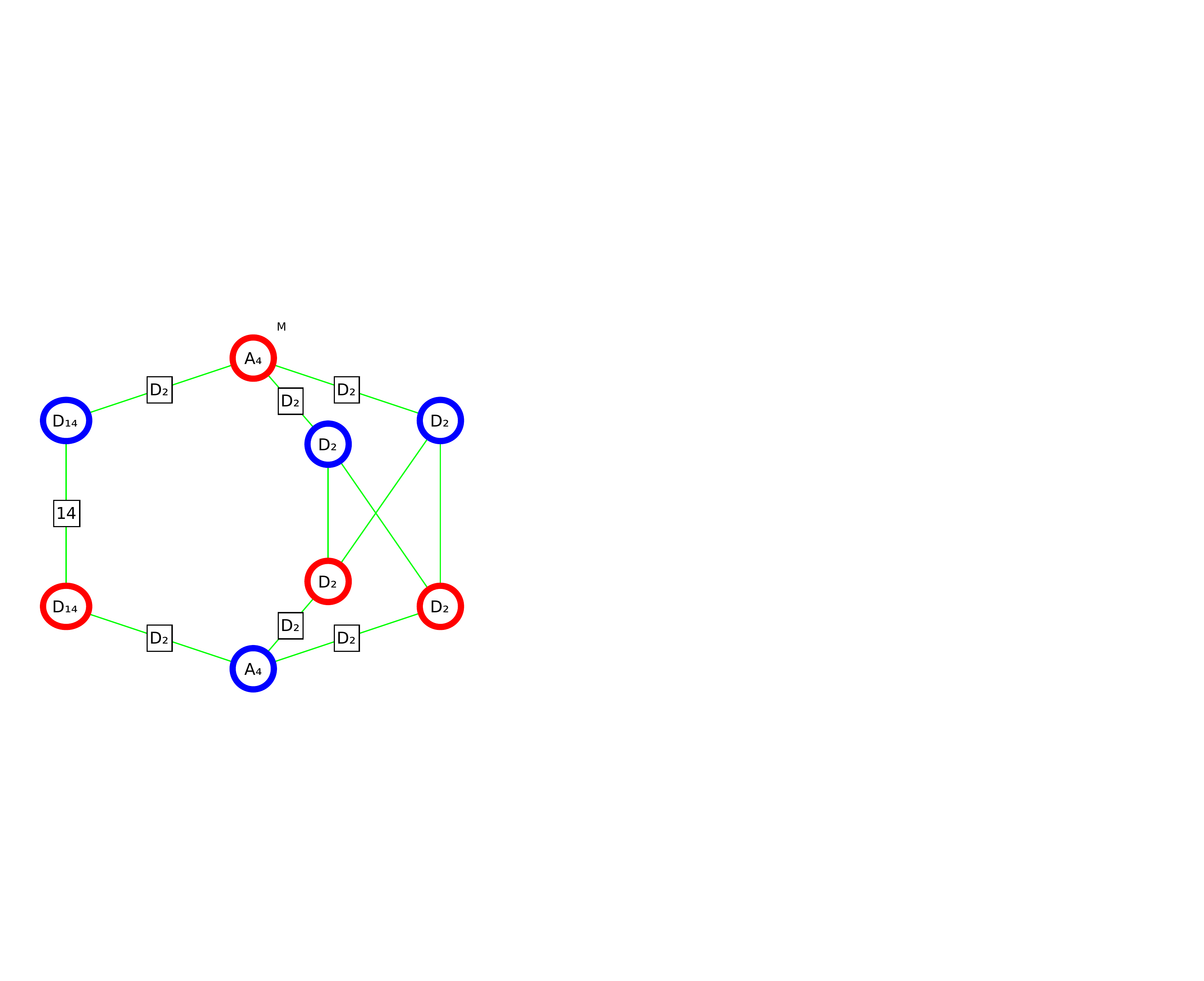}
\end{figure}
\noindent This graph has genus $4$.
By Theorem \ref{theorem: free generators}
\begin{equation}
  \label{ps28}
  \PSUT(R_{28}) \isisom \Gamma_{28,0}\isisom
  \pi_1(\Gr_{28})= A_4 *_{D_2} D_{14} *_{C_{14}}
  D_{14} *_{D_2} A_4 * \Z^{*4}
\end{equation}
and $\chi(\PSUT(R_{28}))=-13/3$.
We can summarize the $n=28$ example with:
\begin{gather}
  \Gg_{28}\ll\UT(R_{28})_f\ll\UTz(R_{28})\ll\UT(R_{28})\\
  \PGg_{28}\ll[\PUT(R_{28})]_f\ll\PUT(R_{28})
\end{gather}
For any $n$ we have $\PUTz(R_n)\lesssim\PUT(R_n)$, so in particular
this is true for $n=28$.  Note that here $[\PUT(R_n)]_f$ is not a subgroup
of $\PUTz(R_n)$: the cyclic group of order $28$ is contained in
$[\PUT(R_n)]_f$ but not in $\PUTz(R_n)$.

\subsection[$n=32$]
    {\texorpdfstring{\protect{\boldmath{$n=32$}}}{n=32}}
    \label{n=32}
We have $\PSUT(R_{32}) = \Gamma_{32, 1} = \Gamma_{32, +}$ and
  $\PUTz(R_{32}) = \PUT(R_{32}) = \Gamma_{32, 0}$.
The quotient graph of groups $\oGr_{32}$ for $\PUTz(R_{32}) = \PUT(R_{32}) = \Gamma_{32,0}$ is shown below broken
into two subgraphs. These subgraphs are to be glued together by
identifying vertices with a label such as $A$ or $\gamma$ in Subgraph 1 with
those with the same label in Subgraph 2. The vertices are also marked
(in the interior of the circle representing the vertex) with their
stabilizers in $\PU_2^\zeta(R_{32})=\PUT(R_{32})$. Recall that an integer $n$ should be 
read as the cyclic group of order $n$.

\begin{figure}[H]
\centering{\scalebox{.92}{\includegraphics[scale = 0.7, trim={
1in -1in -1in -1in},clip]
{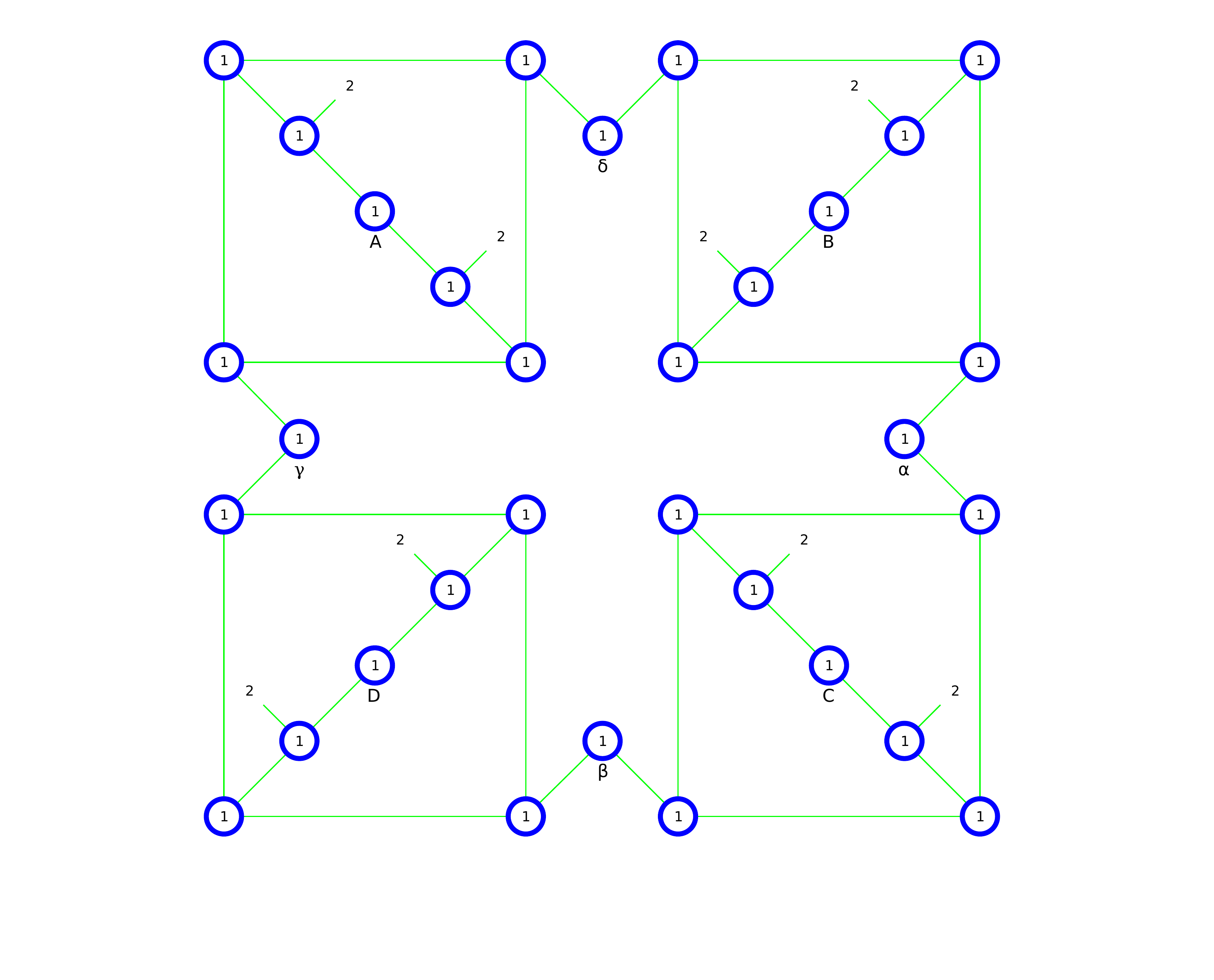}}}
\vspace{-0.45in}

\caption{ Subgraph 1 for $\PUTz(R_{32}) =\PUT(R_{32})=\Gamma_{32,0}$.}
\end{figure}

  \begin{figure}[H]
      
\centering{\scalebox{.92}{\includegraphics[scale = 0.6,
trim={1in -1in 1in -1in}, clip]{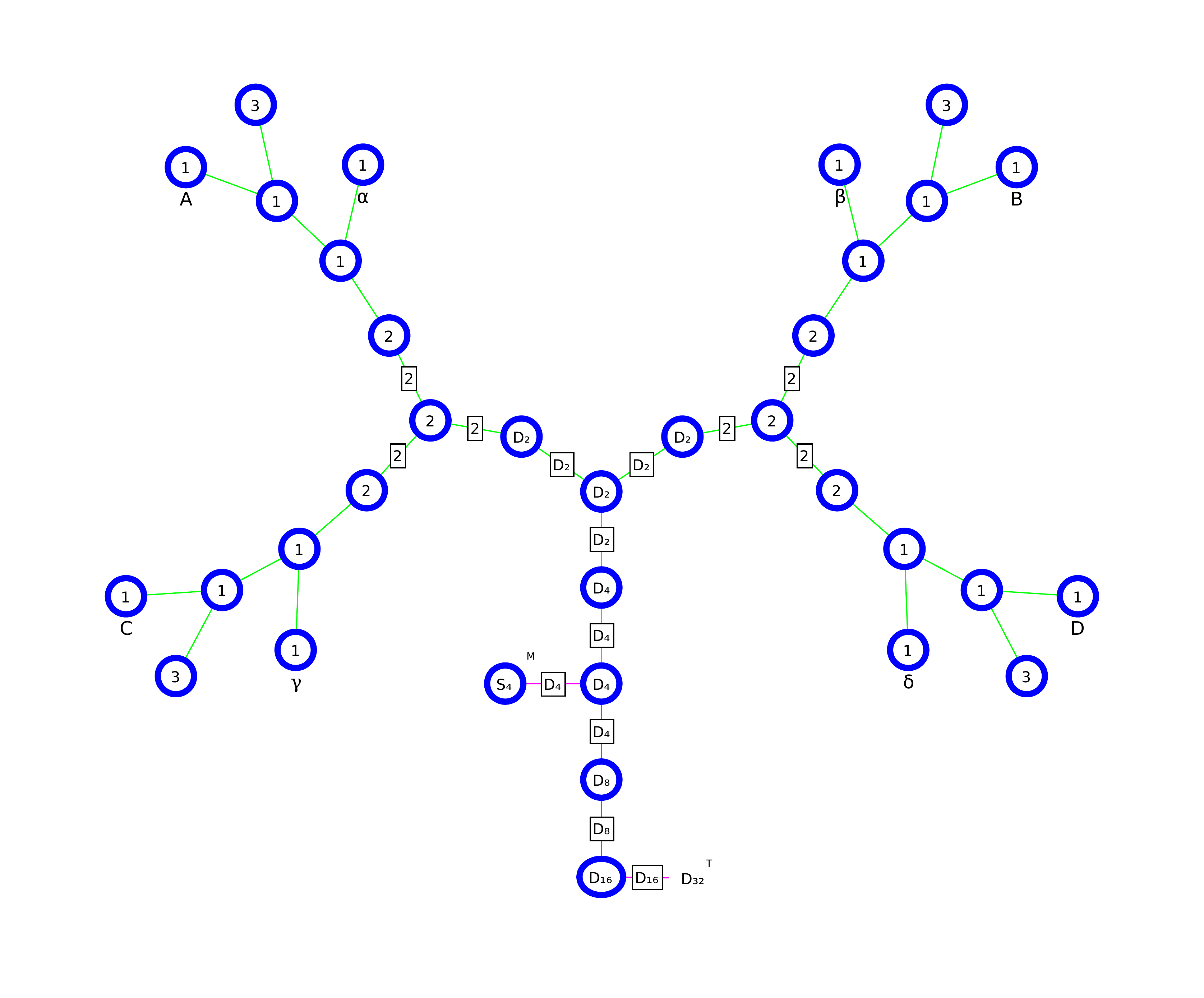}}}

\vspace{-0.1in}

\caption{ Subgraph 2 for $\PUTz(R_{32})= \PUT(R_{32})= \Gamma_{32,0}$.}
\end{figure}

  
 This case has the largest graph, since there are $58$
maximal orders.  On the other hand, $40$ of the maximal orders have
only $\pm 1$ as units.  It also has some edges that join a maximal
order to an isomorphic one, which does not occur for $n = 40, 48$.
The graph has genus $16$. By Theorem \ref{theorem: free generators} we have.
\begin{equation*}
  \PUT(R_{32})=\PU_2^\zeta(R_{32})\isisom \pi_1(\oGr_{32})\isisom
  D_{32} *_{D_4} S_4 * C_3^{*4} *
C_2^{*8} * \Z^{*{16}}
\end{equation*}
and $\chi(\PUT(R_{32}))=-1455/64$.
Again we see that $\PGg_{32} \sinf \PUT(R_{32})_f$, but this time
$\PUT(R_{32})_f \sinf \PUT(R_{32})$.

 The quotient graph of groups $\Gr_{32}$ for $\PSUT(R_{32}) = \Gamma_{32,1} = \Gamma_{32, +}$ is the bipartite double cover of this graph, the
vertex labels being the same.  The maximal orders with nontrivial unit
group again form a forest, and counting vertices and edges we see that
the graph has genus $40$.  Thus, by Theorem \ref{theorem: free generators},
the group is
\begin{equation*}
\PSUT(R_{32})= \Gamma_{32, 1} = \Gamma_{32, +}\isisom \pi_1(\Gr_{32}) \isisom  S_4 *_{D_4}
D_{16} *_{D_4} S_4 * C_3^{*8} * \Z^{*{40}}.
\end{equation*}

\subsection[$n=36$]
{\texorpdfstring{\protect{\boldmath{$n=36$}}}{n=36}}
In this case, $\PSUT(R_{36}) = \Gamma_{36, 1}$ while $\PUTz(R_{36}) =
\PUT(R_{36}) = \Gamma_{36, +} = \Gamma_{36, 0}$.
The quotient graph of groups $\oGr_{36}$  for $\PUTz(R_{36}) =\PUT(R_{36}) = \Gamma_{36, +} = \Gamma_{36, 0}$ is shown below. Notice the doubled edge.

\begin{figure}[H]
  \centering{\scalebox{.92}{\includegraphics[scale = 0.5,
trim={0, 3.5in, 0, 3in}, clip]{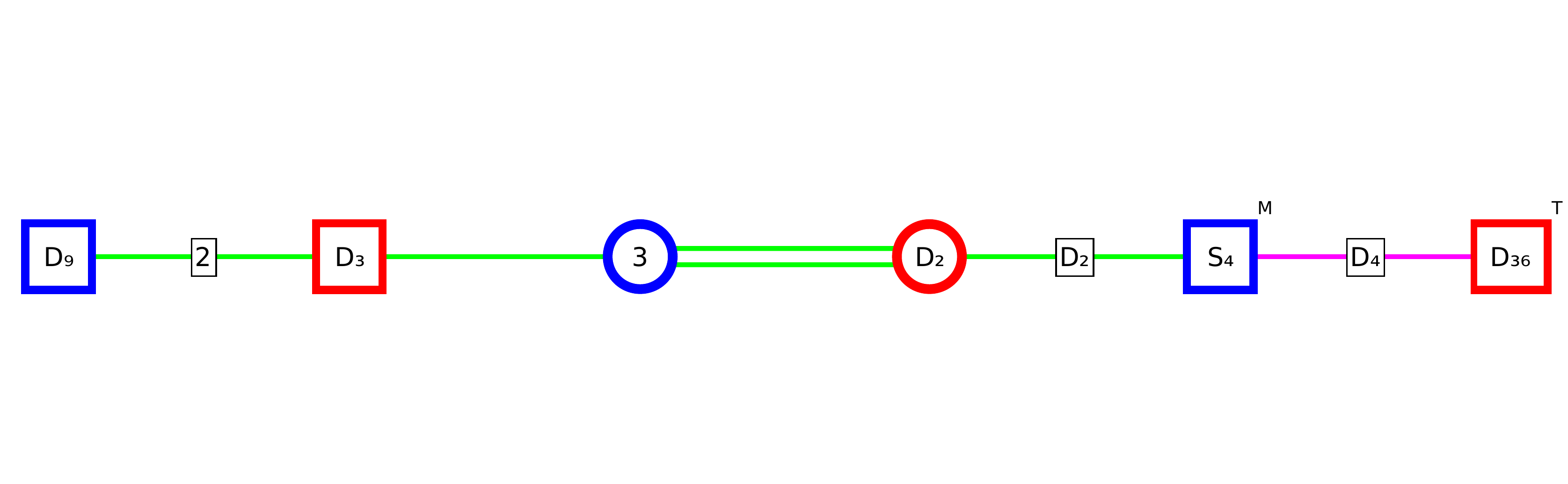}}}

\vspace{-0.1in}

\caption{ Graph of groups for $\PUTz(R_{36})=
  \PUT(R_{36}) = \Gamma_{36, +} = \Gamma_{36, 0}$.}
  \label{dino}
\end{figure}

In this case the residue field of the prime $\fp$ above $2$ has order $8$,
so each maximal order contains $9$ Eichler orders of level $\fp$,
rather than $3$ or $5$ as in the other examples. 
The graph has genus $1$. By Theorem \ref{theorem: free
  generators} we have
\begin{align*}
  \PUTz(R_{36})=&\PUT(R_{36})=\Gamma_{36,0}=\Gamma_{36,+} \isisom\\
  &\pi_1(\oGr_{36})
  \isisom D_9 *_{C_2} D_3 * C_3 * S_4 *_{D_4} D_{36} * \Z
\end{align*}
and $\chi(\PUT(R_{36}))=-217/72$.

As before, the presence of additional factors of finite order implies that
$\PGg_{36} \sinf [\PUT(R_{36})]_f$, while 
$[\PUT(R_{36})]_f \sinf \PUT(R_{36})$ because of the $\Z$ in the list of
factors  .

 The quotient graph of groups $\Gr_{36}$ for $\PSUT(R_{36}) = \Gamma_{36,1}$ is a double cover of $\oGr_{36}$ ramified at the
 four vertices indicated by squares in Figure \ref{dino}.
 The vertex labels for the unramified nodes are
the same, and the ramified nodes have vertices labeled with subgroups
of index~$2$.
The graph has genus $3$.  By Theorem \ref{theorem: free
   generators} we see that
\begin{equation*}
\PSUT(R_{36})= \Gamma_{36,1} \isisom \pi_1(\Gr_{36})\isisom C_9 * C_3^{*3} * A_4 *_{D_2} D_{18} * \Z^{*3}.
\end{equation*}

\subsection[$n=40$]
{\texorpdfstring{\protect{\boldmath{$n=40$}}}{n=40}}
Here $\PSUT(R_{40}) = \Gamma_{40, 1}$, while $\PUTz(R_{40}) =
\PUT(R_{40}) = \Gamma_{40, +} = \Gamma_{40, 0}$.
The quotient graph of groups $\oGr_{40}$ for $\PUTz(R_{40}) =\PUT(R_{40}) = \Gamma_{40, +} = \Gamma_{40, 0}$ is shown below broken into two subgraphs. The two
subgraphs are to be glued by identifying vertices with the same
label, e.g., Vertex $A$ in Subgraph 1 is identified with Vertex $A$ in
Subgraph 2. 

\begin{figure}[H]

\moveright.3in\vbox{\scalebox{.90}{\includegraphics[scale = 0.65,
trim={1.6in -1.7in 0 -1.6in},clip]{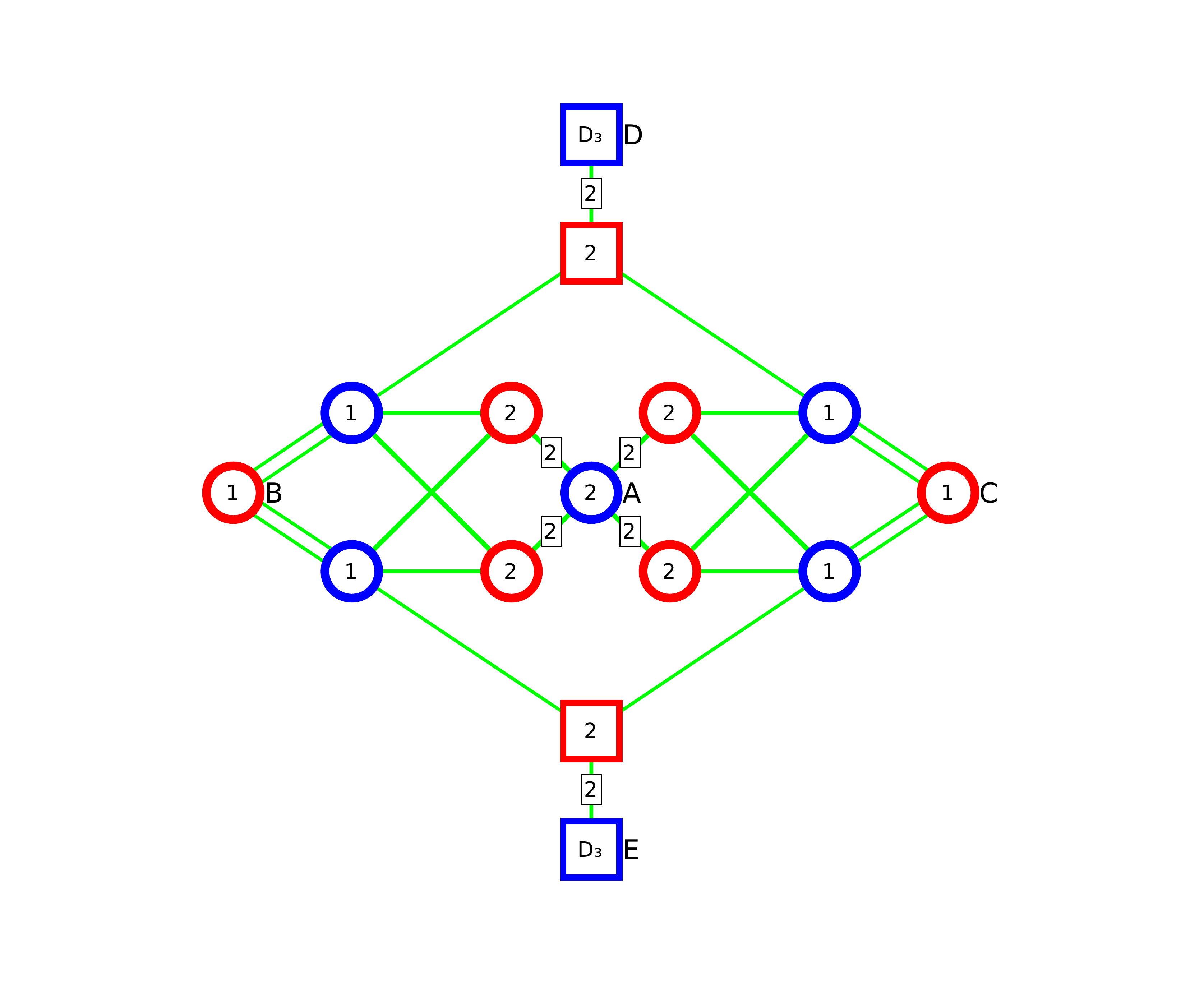}}}

\vspace{-0.4in}

\caption{ Subgraph 1 for $\PUTz(R_{40}) =\PUT(R_{40}) = \Gamma_{40, +} = \Gamma_{40, 0}$.}
\label{pea1}
\end{figure}

\begin{figure}[H]
  
\moveright.2in\vbox{\scalebox{.90}{\includegraphics[scale = 0.6,
trim={1.2in -2in 1in -2in}, clip]{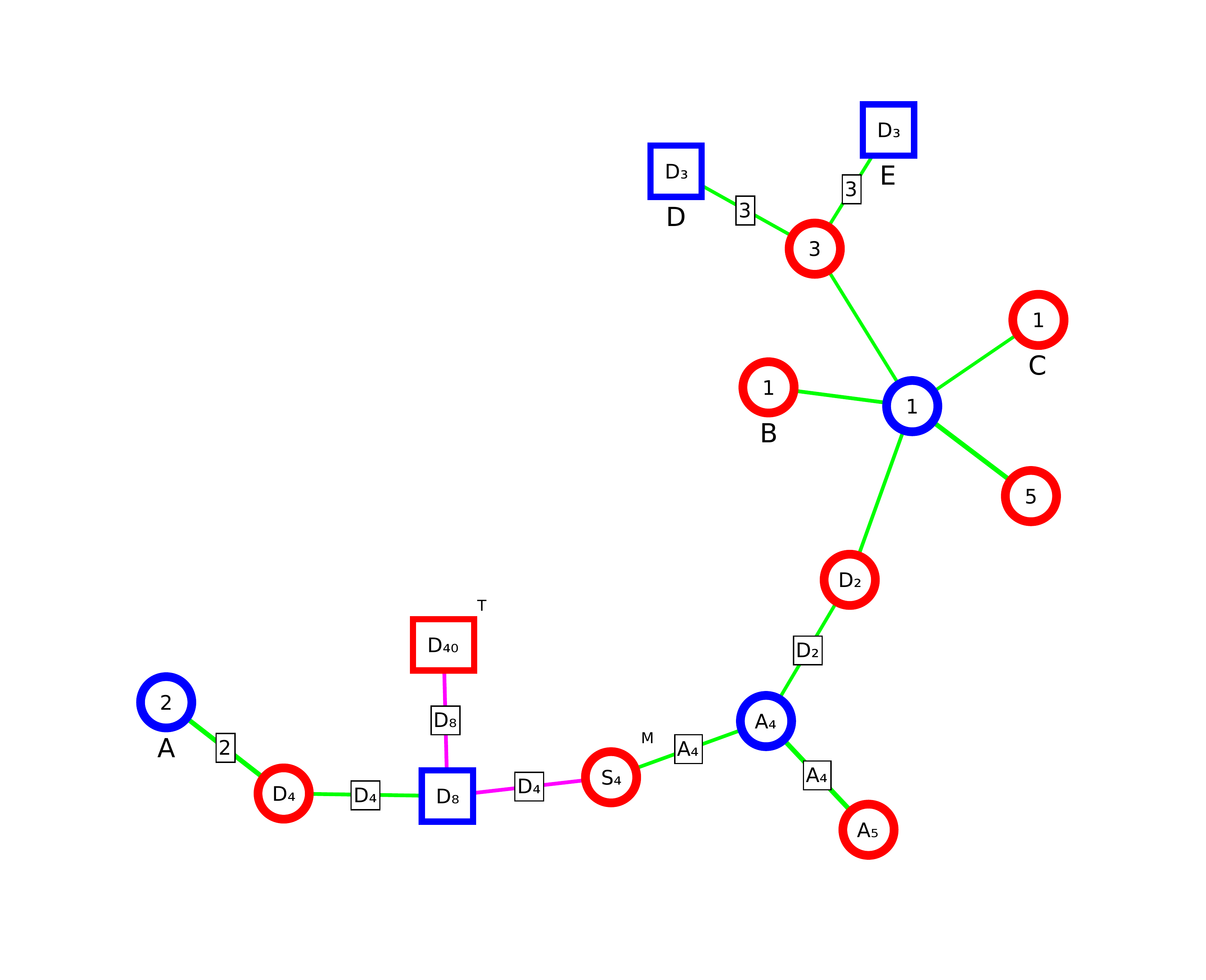}}}

\vspace{-0.4in}

\caption{ Subgraph 2 for $\PUTz(R_{40}) =\PUT(R_{40}) = \Gamma_{40, +} = \Gamma_{40, 0}$.}
\label{pea2}
\end{figure}

 In this case the residue field has order $4$, so each maximal order
contains five Eichler orders rather than three.  This means that the
graph is more highly connected than in the other cases 
with $\phi(n) = 16$.

 The automorphism group of each edge is the same as the smaller
automorphism group of its nodes except for the edge between the nodes
with automorphism groups $D_8$ and $S_4$ which has automorphism group
$D_4$.

 The graph again has genus $16$.  Theorem \ref{theorem: free
   generators} gives
\begin{align*}
  \PUT(R_{40})=& \PU_2^\zeta(R_{40})\isisom\pi_1(\oGr_{40})
  \isisom\\
 & D_{40} *_{D_4} S_4 *_{A_4} A_5 * D_3 *_{C_3} D_3 * C_5 * \Z^{*{16}}
\end{align*}
and $\chi(\PUT(R_{40}))=-287/16$.
Once again we have
\[
\PGg_{40} \sinf [\PUT(R_{40})]_f \sinf \PUT(R_{40}) .
\]

The graph $\Gr_{40}$ for $\PSUT(R_{40}) = \Gamma_{40, 1}$ is a double cover
of $\oGr_{40}$
ramified at the
six vertices in Figures \ref{pea1}, \ref{pea2} indicated by squares.
The vertex labels for the unramified nodes are
the same, and the ramified nodes have vertices labeled with subgroups
of index~$2$.  This case is different in that the subgraph of
nontrivial unit groups has a loop, specifically a square all whose
vertices have group $C_3$ (the remaining components are all trees).
Since $\PUT(R_{40})/\PSUT(R_{40})$ acts by reflection on this square,
the monodromy is trivial.

 The whole graph has genus $34$.  Theorems \ref{theorem: loop} and \ref{theorem: free
   generators} show
\begin{align*}
\PSUT(R_{40})&\isisom \Gamma_1 \isisom \pi_1(\Gr_{40})\\&\isisom
 A_5 *_{A_4} S_4 *_{D_4} D_{20} *_{D_4} S_4 *_{A_4} A_5 *  C_5^{*2} *
(C_3 \oplus \Z) * \Z^{*{33}}.
\end{align*}

\subsection[$n=48$]
    {\texorpdfstring{\protect{\boldmath{$n=48$}}}{n=48}}
    \label{ready}
Again $\PSUT(R_{48}) = \Gamma_{48, 1}$ while $\PUTz(R_{48}) =\PUT(R_{48}) = \Gamma_{48, +} = \Gamma_{48, 0}$.
The quotient graph of groups $\oGr_{48}$
for $\PUTz(R_{48}) =\PUT(R_{48}) = \Gamma_{48, +} = \Gamma_{48, 0}$ is shown in the same format as for $n=40$.
   
  \begin{figure}[H]
    
\moveright.3in\vbox{\scalebox{.90}{\includegraphics[scale = 0.65,
trim={1.2in -1in 1in -1.7in}, clip]{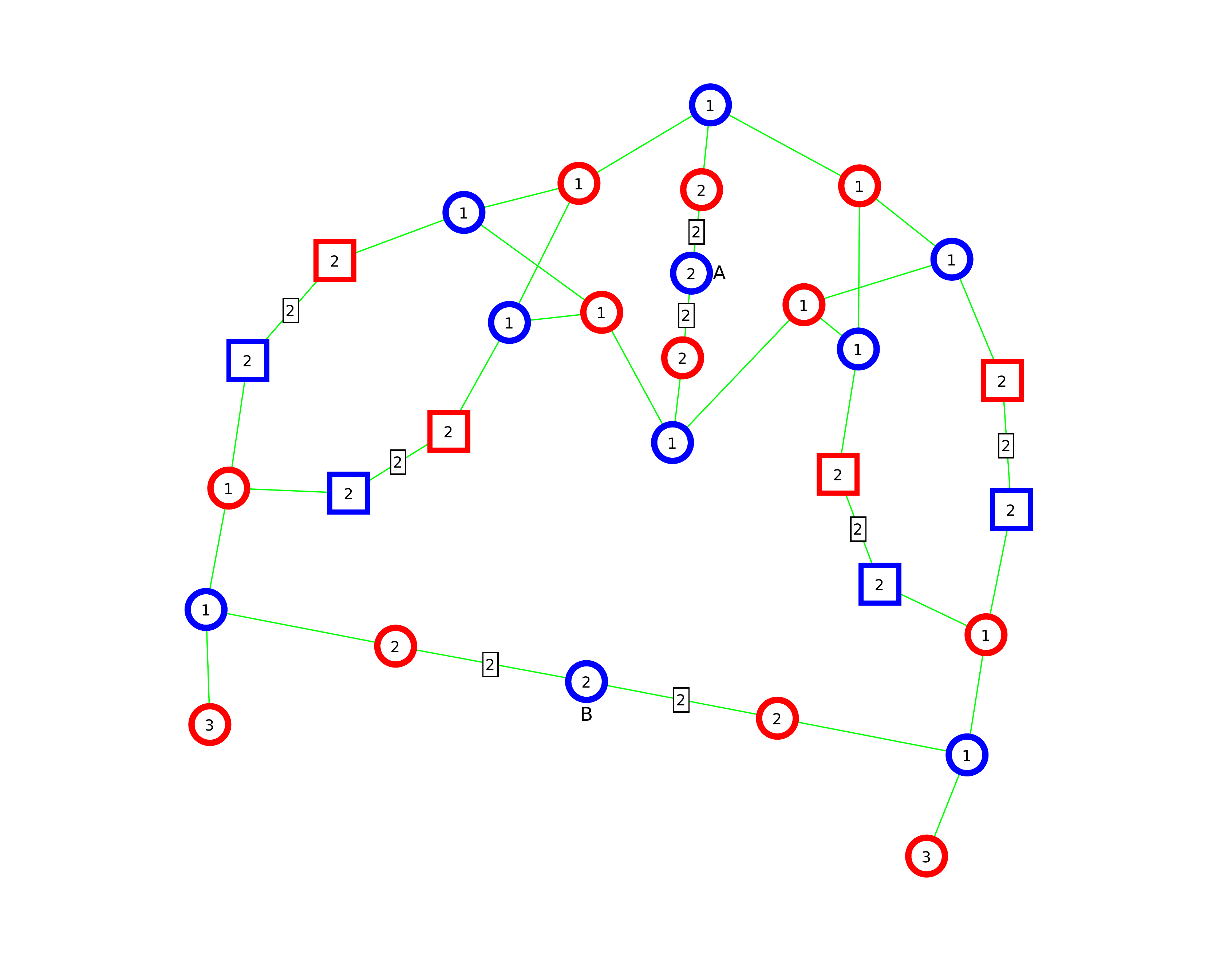}}}

\caption{ Subgraph 1 for $\PUTz(R_{48}) =\PUT(R_{48}) = \Gamma_{48, +} = \Gamma_{48, 0}$.}
\label{bean1}
\end{figure}

\begin{figure}[H]

\moveright.3in\vbox{\scalebox{.90}{\includegraphics[scale = 0.7,
trim={1.2in -1.4in .8in -1.3in},clip]{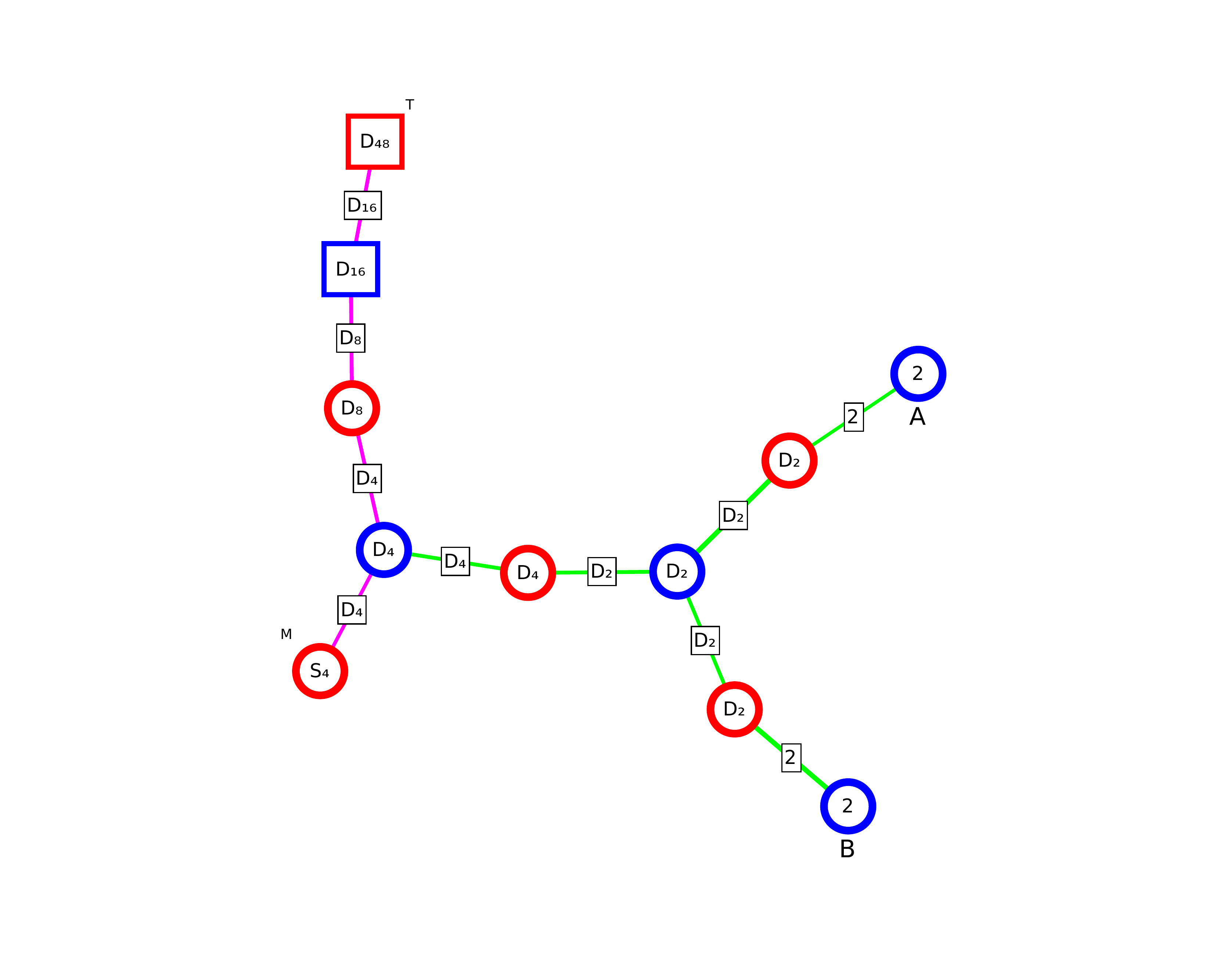}}}

\vspace{-0.3in}

\caption{ Subgraph 2 for $\PUTz(R_{48}) =\PUT(R_{48}) = \Gamma_{48, +} = \Gamma_{48, 0}$.}
\label{bean2}
\end{figure}

 There are $39$ classes
of maximal orders, but of these $14$ have only $\pm 1$ as units.
Modulo $\pm 1$, there are also $14$ orders with $2$ units, three with $4$,
two each with $3$ and $8$, and one each with $16, 24, 32, 96$.
The automorphism group of each edge is the same as the smaller
of the two automorphism groups of its incident nodes.

 The graph has genus $8$, and Theorem \ref{theorem: free
   generators} and the diagram indicate that
\begin{equation*}
\PUTz(R_{48})=\PUT(R_{48})\isisom \pi_1(\oGr_{48})\isisom D_{48}
*_{D_4} S_4 * C_3^{*2} * C_2^{*4} * \Z^{*8}
\end{equation*}
and $\chi(\PUT(R_{48}))=-365/32$.
For the same reason as in the cases $n = 32, 36, 40$, this presentation
shows that
\begin{equation}
  \label{mustard}
  \PGg_{48} \sinf [\PUT(R_{48})]_f \sinf \PUT(R_{48}).
  \end{equation}

 The graph of groups $\Gr_{48}$ for $\PSUT(R_{48}) = \Gamma_{48, 1}$ is a double cover of $\oGr_{48}$  ramified at the
 ten vertices indicated by squares in Figures \ref{bean1}, \ref{bean2}.
 As before the vertex labels for the
unramified nodes are the same, and the ramified nodes have vertices
labeled with subgroups of index~$2$.  Counting vertices and
edges we see that the graph has genus $20$.
Thus by Theorem \ref{theorem: free generators} the group is
\begin{equation*}
\PSUT(R_{48}) \isisom \pi_1(\Gr_{48})\cong S_4 *_{D_4}
D_{48} *_{D_4} S_4 * C_3^{*4} * \Z^{*{20}}.
\end{equation*}

\subsection[$n=60$]
           {\texorpdfstring{\protect{\boldmath{$n=60$}}}{n=60}}
As with $n=28$, the
unique prime above $2$ in $F_n$ splits in $K_n$.
Hence
$\PSUT(R_{60}) = \Gamma_{60,1}$, $\PUTz(R_{60}) =
\Gamma_{60,+}$, and $\PUT(R_{60}) = \Gamma_{60,0}$ are all distinct with
\[
\frac{\PUT(R_{60})}{\PUTz(R_{60})} \isisom \Z/2\Z\qquad\text{and}\qquad
  \frac{\PUTz(R_{60})}{\PSUT(R_{60})} \isisom \Z/2\Z .
  \]

 The class number of $\H_{60}$ is $9$, but as in the $n=28$ case not all
 types are connected to $\M$. In this case, 7 of the 9 types occur in
 this (genus 5) quotient h-graph of groups
$\oGr_{60}$
 for $\PUT(R_{60}) = \Gamma_{60, 0}$.

\begin{figure}[H]
    
\includegraphics[scale = 0.7,
trim={.15in -.7in 1.35in -.3in}, clip]{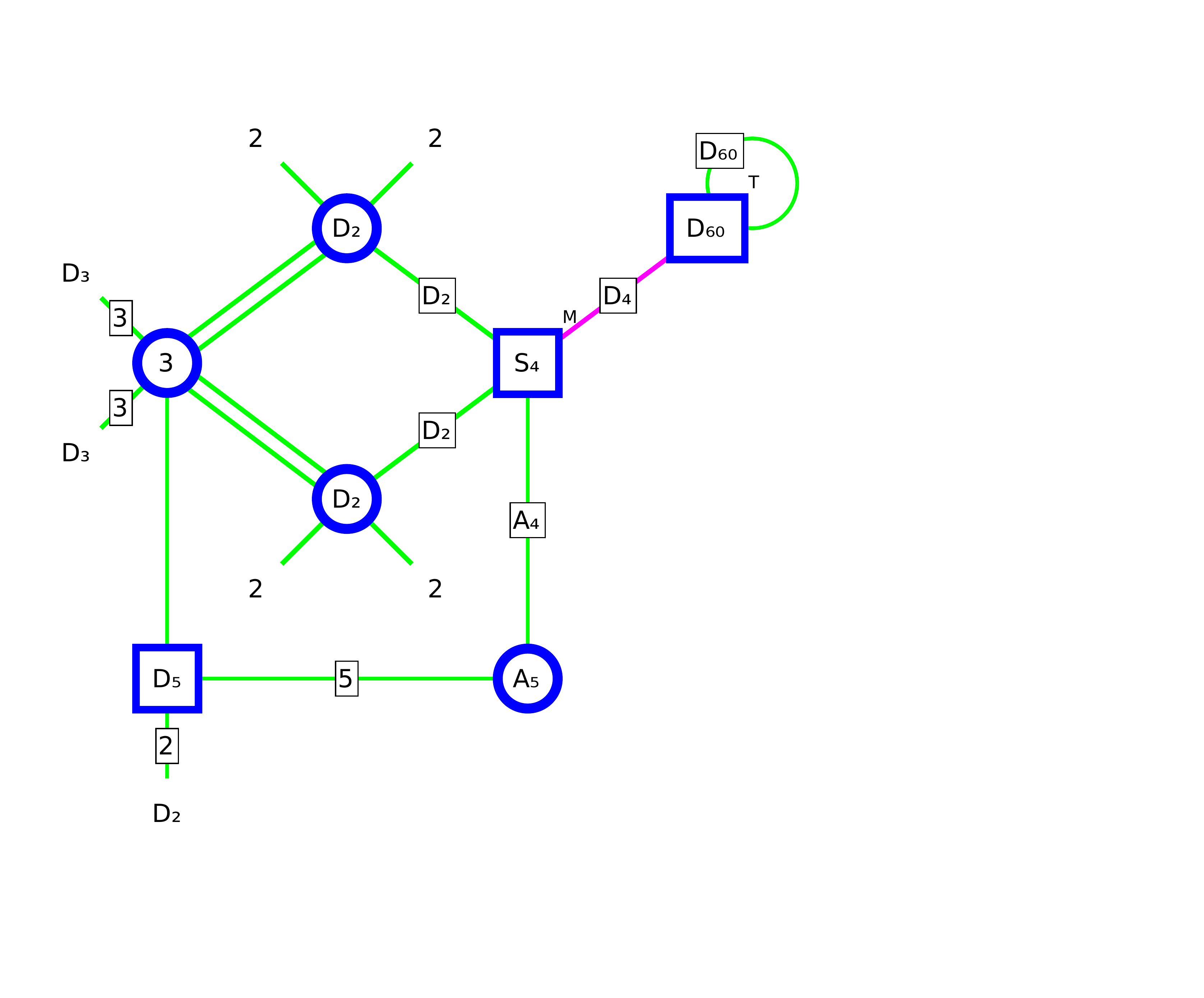}

\vspace{-0.5in}
   
\caption{ Graph of groups $\oGr_{60}$ for $\PUT(R_{60}) =\Gamma_{60,0}$.}
\label{caul}
\end{figure}

Let $\Gr_u$ be the graph of groups that is a single loop in the upper
right whose vertex and edge groups are both $D_{60}$.  As it turns out
the induced automorphism on $D_{60}$ is inner, so by Theorem
\ref{theorem: loop}, $\pi_1(\Gr_u) = D_{60}\oplus \Z$. Let $\Gr_l$ be
the graph of groups which remains  after deleting $Gr_u$ and the $D_4$
edge incident upon it. Clearly,
$$\PUT(R_{60}) = \pi_1(\Gr_u)*_{D_4} \pi_1(\Gr_l)\, .$$
Now, let $T$ be the spanning tree for $\Gr_l$
obtained by eliminating the four edges with trivial stabilizer groups
incident upon the vertex with vertex group $C_3$. Two each of these
edges are incident upon each of the two vertices with vertex groups
$D_2$. Theorem \ref{theorem: free generators} then tells us that
$$\pi_1(\Gr_l) \isisom S_4
  *_{A_4} A_5 *_{C_5} D_5 *_{C_2} D_2 * D_3 * _{C_3} * D_3 * C_2^{*4}
  * \Z^{*4}\, ,$$
so that
\begin{equation*}
  \PUT(R_{60}) \isisom (D_{60} \oplus \Z) *_{D_4} S_4
  *_{A_4} A_5 *_{C_5} D_5 *_{C_2} D_2 * D_3 * _{C_3} * D_3 * C_2^{*4}
  * \Z^{*4}
\end{equation*}
and $\chi(\PUT(R_{60}))=-15/2$.

 The graph for $\PUTz(R_{60})=\Gamma_{60,+}$ is the bipartite double cover of the
 graph $\oGr_{60}$ for $\PUT(R_{60})=\Gamma_{60,0}$.  The graph
$\Gr_{60}$
 for $\PSUT(R_{60})=\Gamma_{60,1}$ is the double
 cover of the graph for $\PUTz(R_{60}) = \Gamma_{60,+}$ ramified at the
 vertices lying above those in $\oGr_{60}$ marked
with a square in Figure \ref{caul}.

\section{Summary}
\label{summary}

We summarize our results in the following table, which shows, for each
group at each level, whether it is generated by torsion.

\begin{center}
  \begin{tabular}{r|c|c|c|c|c|c}
    $n$ & $\PSUT$ & $\PU_2^\zeta$ & $\PUT$
    & $\Gamma_1$ & $\Gamma_+$ & $\Gamma_0$\\ \hline
    $8$ & yes & yes & yes & yes & yes & yes\\ \hline
    $12$ & yes & yes & yes & yes & yes & yes\\ \hline
    $16$ & yes & yes & yes & yes & yes & yes\\ \hline
    $20$ & yes & yes & yes & yes & yes & yes\\ \hline
    $24$ & yes & yes & yes & yes & yes & yes\\ \hline
    $28$ & no & no & yes & no & no & yes\\ \hline
    $32$ & no & no & no & no & no & no\\ \hline
    $36$ & no & no & no & no & no & no\\ \hline
    $40$ & no & no & no & no & no & no\\ \hline
    $48$ & no & no & no & no & no & no\\ \hline
    $60$ & no & no & no & no & no & no
  \end{tabular}
\end{center}

\clearpage

\bibliography{Paper_2_GIT}{}

\begin{bibdiv}
\begin{biblist}

\bib{BCP}{article}{
      author={Bosma, Wieb},
      author={Cannon, John},
      author={Playoust, Catherine},
       title={The {M}agma algebra system. {I}. {T}he user language},
        date={1997},
        ISSN={0747-7171},
     journal={J. Symbolic Comput.},
      volume={24},
      number={3-4},
       pages={235\ndash 265},
      review={\MR{1484478}},
}

\bib{FGKM}{article}{
      author={Forest, Simon},
      author={Gosset, David},
      author={Kliuchnikov, Vadym},
      author={McKinnon, David},
       title={Exact synthesis of single qu-bit unitaries over
  {C}lifford-cyclotomic gate sets},
        date={2015},
        ISSN={0022-2488},
     journal={Journal of Mathematical Physics},
      volume={56},
      number={8},
       pages={082201, 26},
      review={\MR{3455329}},
}

\bib{I}{article}{
      author={Ihara, Y.},
       title={On discrete subgroups of the two by two projective linear group
  over $\mathfrak{p}$-adic fields.},
        date={1966},
        ISSN={0025-5625},
     journal={J. Math. Soc. Japan},
      volume={18},
      number={3},
       pages={219\ndash 235},
      review={\MR{0223463}},
}

\bib{IJKLZ}{article}{
      author={Ingalls, Colin},
      author={Jordan, Bruce~W.},
      author={Keeton, Allan},
      author={Logan, Adam},
      author={Zaytman, Yevgeny},
       title={The {C}lifford-cyclotomic group and {E}uler-{P}oincar\'{e}
  characteristics},
        date={2019},
        note={arXiv: 1903.09497},
}

\bib{IJKLZ2}{article}{
      author={Ingalls, Colin},
      author={Jordan, Bruce~W.},
      author={Keeton, Allan},
      author={Logan, Adam},
      author={Zaytman, Yevgeny},
       title={The corank of unitary groups over cyclotomic rings},
        date={2019},
        note={arXiv:1911.02137},
}

\bib{JL}{article}{
      author={Jordan, Bruce~W.},
      author={Livn{\'e}, Ron},
       title={The {R}amanujan property for regular cubical complexes},
        date={2000},
        ISSN={0012-7094},
     journal={Duke Mathematical Journal},
      volume={105},
      number={1},
       pages={85\ndash 103},
      review={\MR{1788043}},
}

\bib{Kn}{incollection}{
      author={Kneser, Martin},
       title={Strong approximation},
        date={1966},
   booktitle={Algebraic {G}roups and {D}iscontinuous {S}ubgroups
  \textup{(}{P}roc. {S}ympos. {P}ure {M}ath., {B}oulder {C}olo.,
  \textup{1965)}},
   publisher={Amer. Math. Soc., Providence, RI},
       pages={187\ndash 196},
      review={\MR{0213361}},
}

\bib{K}{article}{
      author={Kurihara, Akira},
       title={On some examples of equations defining {S}himura curves and the
  {M}umford uniformization},
        date={1979},
        ISSN={0040-8980},
     journal={J. Fac. Sci. Univ. Tokyo Sect. IA Math},
      volume={25},
      number={3},
       pages={277\ndash 300},
      review={\MR{523989}},
}

\bib{LS}{book}{
      author={Lyndon, Roger~C.},
      author={Schupp, Paul~E.},
       title={Combinatorial group theory},
      series={Classics in Mathematics},
   publisher={Springer-Verlag, Berlin},
        date={2001},
        ISBN={3-540-41158},
      review={\MR{1812024}},
}

\bib{RS}{article}{
      author={Radin, Charles},
      author={Sadun, Lorenzo},
       title={On \textup{2}-generator {S}ubgroups of {${\rm SO}(3)$}},
        date={1999},
     journal={Trans. {AMS}},
      volume={351},
      number={11},
       pages={4469\ndash 4480},
}

\bib{S}{article}{
      author={Sarnak, Peter},
       title={Letter to {S}cott {A}aronson and {A}ndy {P}ollington on the
  {S}olovay-{K}itaev theorem and golden gates},
        date={2015},
      eprint={http://publications.ias.edu/sarnak/paper/2637},
}

\bib{S2}{book}{
      author={Serre, Jean-Pierre},
       title={Trees},
      series={Springer Monographs in Mathematics},
   publisher={Springer-Verlag, Berlin},
        date={2003},
        ISBN={3-540-44237-5},
        note={Translated from the French original by John Stillwell, Corrected
  2nd printing of the 1980 English translation},
      review={\MR{1954121}},
}

\bib{S3}{inproceedings}{
      author={Serre, Jean-Pierre},
       title={Cohomologie des groupes discrets},
        date={1971},
   booktitle={Prospects in mathematics ({P}roc. {S}ympos., {P}rinceton {U}niv.,
  {P}rinceton, {N}.{J}., 1970)},
       pages={77\ndash 169, Ann. of Math. Studies, No. 70},
}

\bib{V}{book}{
      author={Vign{\'e}ras, Marie-France},
       title={Arithm\'etique des alg\`ebres de quaternions},
      series={Lecture Notes in Mathematics},
   publisher={Springer, Berlin},
        date={1980},
      volume={800},
        ISBN={3-540-09983-2},
      review={\MR{580949}},
}

\bib{Vo}{misc}{
      author={Voight, John},
       title={Quaternion algebras},
         how={{https://math.dartmouth.edu/~jvoight/quat.html}},
        date={2019},
}

\bib{We}{book}{
      author={Weber, H.},
       title={Lehrbuch der {A}lgebra, {V}ol. \textup{II}, zweite {A}uflage},
   publisher={Vieweg, Braunschweig},
        date={1899},
}

\end{biblist}
\end{bibdiv}
\bibliographystyle{plain}

\end{document}